\documentclass[12pt]{amsart}

\usepackage{tikz,amsthm,amsmath,amstext,amssymb,amscd,epsfig,euscript, mathrsfs, dsfont,pspicture,multicol,graphpap,graphics,graphicx,times,enumerate,subfig,sidecap,wrapfig,color,pict2e}

\usepackage[colorlinks,citecolor=red,pagebackref,hypertexnames=false]{hyperref}

\usepackage{enumerate}

 \numberwithin{equation}{section}

\def\bB{{\mathbb{B}}}
\def\bC{{\mathbb{C}}}

\def\bR{{\mathbb{R}}}
\def\R{{\mathbb{R}}}

\def\bZ{{\mathbb{Z}}}

\def\cB{{\mathscr{B}}}

\def\cD{{\mathscr{D}}}

\def\cF{{\mathscr{F}}}
\def\cG{{\mathscr{G}}}
\def\cH{{\mathscr{H}}}

\def\cT{{\mathscr{T}}}

\def\cW{{\mathscr{W}}}

\def\one{\mathds{1}}

\def\ve{\varepsilon}

\renewcommand{\d}{{\partial}}
\def\lec{\lesssim}
\def\gec{\gtrsim}

\DeclareMathOperator{\diam}{diam}
\def\dist{\mathop\mathrm{dist}} 						
\def\supp{\mathop\mathrm{supp}}					

\newcommand{\ps}[1]{\left( #1 \right)}

\newcommand{\ck}[1]{\left\{#1 \right\}}

\newcommand{\floor}[1]{\left\lfloor #1 \right\rfloor}

\newcommand{\cnj}[1]{\overline{#1}}
\newcommand{\isif}[1]{\left\{\begin{array}{cc} #1
\end{array}\right.}

\def\warrow{\rightharpoonup}

\def\XXint#1#2#3{{\setbox0=\hbox{$#1{#2#3}{\int}$ }
\vcenter{\hbox{$#2#3$ }}\kern-.58\wd0}}

\theoremstyle{plain}
\newtheorem{theorem}{Theorem}

\newtheorem{lemma}[theorem]{Lemma}

\theoremstyle{definition}

\newtheorem{definition}[theorem]{Definition}
\newtheorem{remark}[theorem]{Remark}

\numberwithin{equation}{section}
\numberwithin{theorem}{section}

\newtheorem{main}{Theorem}

\def\TheoremI{Theorem \ref{thmi} }

\newtheorem{prop}{Proposition}

  \DeclareFontFamily{U}{mathb}{\hyphenchar\font45} 
\DeclareFontShape{U}{mathb}{m}{n}{
      <5> <6> <7> <8> <9> <10> gen * mathb
      <10.95> mathb10 <12> <14.4> <17.28> <20.74> <24.88> mathb12
      }{}
\DeclareSymbolFont{mathb}{U}{mathb}{m}{n}

\DeclareMathSymbol{\toitself}{3}{mathb}{"FD}  

\begin{document}

\title[Semi-uniform domains and the $A_{\infty}$ property]{Semi-uniform domains and the $A_{\infty}$ property for harmonic measure}
\author{Jonas Azzam}
\address{School of Mathematics, University of Edinburgh, JCMB, Kings Buildings,
	Mayfield Road, Edinburgh,
	EH9 3JZ, Scotland.}
\email{j.azzam "at" ed.ac.uk}

\begin{abstract}
We study the properties of harmonic measure in semi-uniform domains. Aikawa and Hirata showed in \cite{AH08} that, for John domains satisfying the capacity density condition (CDC), the doubling property for harmonic measure is equivalent to the domain being semi-uniform. Our first result removes the John condition by showing that any domain satisfying the CDC whose harmonic measure is doubling is semi-uniform. Next, we develop a substitute for some classical estimates on harmonic measure in nontangentially accessible domains that works in semi-uniform domains. 

We also show that semi-uniform domains with uniformly rectifiable boundary have big pieces of chord-arc subdomains. We cannot hope for big pieces of Lipschitz subdomains (as was shown for chord-arc domains by David and Jerison  \cite{DJ90}) due to an example of Hrycak, which we review in the appendix.

Finally, we combine  these tools to study the $A_{\infty}$-property of harmonic measure. For a domain with Ahlfors-David regular boundary, it was shown by Hofmann and Martell that the $A_{\infty}$ property of harmonic measure implies uniform rectifiability of the boundary \cite{HM15,HLMN17} . Since $A_{\infty}$-weights are doubling, this also implies the domain is semi-uniform. Our final result shows that these two properties, semi-uniformity and uniformly rectifiable boundary, also imply the $A_{\infty}$ property for harmonic measure, thus classifying geometrically all domains for which this holds. 
\end{abstract}


\maketitle

\tableofcontents

\section{Introduction}
In this paper we study a few connections between the geometry of a domain $\Omega\subseteq \R^{d+1}$ and the behavior of its harmonic measure $\omega_{\Omega}^{x}$ with pole $x\in \Omega$. Our motivation is to obtain a  characterization of the $A_{\infty}$-property for harmonic measure on the boundary, however this first  requires understanding how the connectivity properties of a domain relate to the doubling properties of harmonic measure. Below we define a few common notions of connectivity that are studied in this context.

\begin{definition}
Let $\Omega\subseteq \bR^{d+1}$ be an open set. 
\begin{enumerate}
	\item For $x,y\in \cnj{\Omega}$, we say a curve $\gamma\subseteq \cnj{\Omega}$ is a {\bf $C$-cigar curve} from $x$ to $y$ if $\min\{\ell(x,z),\ell(y,z)\}\leq C \dist(z,\Omega^{c})$ for all $z\in \gamma$, where $\ell(a,b)$ denotes the length of the sub-arc in $\gamma$ between $a$ and $b$.  We will also say it has {\bf bounded turning} if  $\ell(\gamma)\leq C |x-y|$. 
\item If there is $x\in \Omega$ such that every $y\in \Omega$ is connected to $x$ by a curve $\gamma$ so that $\ell(y,z)\leq C \dist(z,\Omega^{c})$ for all $z\in \Gamma$,  we say $\Omega$ is {\bf $C$-John}.
\item If every pair $x\in \Omega$ and $\xi\in \d\Omega$ are connected by a $C$-cigar with bounded turning, then we say $\Omega$ is {\bf $C$-semi-uniform (SU)}.
\item If every $x,y\in \Omega$ are connected by a $C$-cigar of bounded turning, we say $\Omega$ is {\bf uniform}. 
\item For a ball $B$ of radius $r_{B}$ centered on $\d\Omega$, we say $x\in B$ is an {\bf interior/exterior $c$-corkscrew point} for $\Omega$ if $B(x,2cr_{B})\subseteq B\cap \Omega$ (or $B(x,2cr_{B})\subseteq B\backslash \Omega$) . We say $\Omega$ satisfies the interior {\bf $c$-Corkscrew condition} if every ball $B$ on $\d\Omega$ has a interior (or exterior) $c$-corkscrew point.
\item A uniform domain with exterior corkscrews is {\bf nontangentially accessible (NTA)}. 
\end{enumerate}	
\end{definition}

Note that a John domain is necessarily bounded. As mentioned in \cite{AH08}, these domains have the following containments:

\[
\mbox{NTA $\subsetneq$ Uniform $\subsetneq$ Semi-uniform $\subsetneq$ John}
\]
where the last containment is only true for bounded domains, and each of these containments can be strict. For example, the complement of a $4$-corner cantor set in $\R^{2}$ is uniform but not NTA. If we set $\Omega=\{(x,y)\in \bR^{d}\times \bR: y\neq 0 \mbox{ or }|x|< 1\}$, then this is semi-uniform but not uniform. A bounded example of a non-uniform semi-uniform domain is $\Omega= B(0,1)\backslash [-1/2,1/2]\subseteq \bC$, see Figure \ref{f:examples}.b. Note that each point along the segment is easily accessible from any other point in the domain by a curve of bounded turning, but points close to the segment and on opposite sides are not. If $\Omega=\bB\backslash [0,1]\subseteq \bC$, then $\Omega$ is John but not semi-uniform, as points in the bottom corner do not easily access boundary points above and near the top corner, see Figure \ref{f:examples}.a. 

Uniform domains were introduced independently by Martio and Sarvas \cite{MS79} and by Jones \cite{Jones-extension-theorems-for-bmo}. To our knowledge, semi-uniform domains were first defined (and only mentioned) by Aikawa and Hirata in \cite{AH08}.

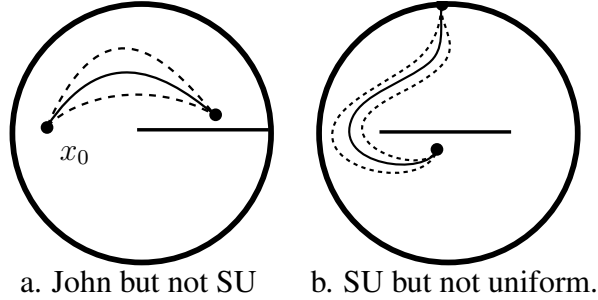
\begin{figure}[!ht]
\includegraphics[width=100pt]{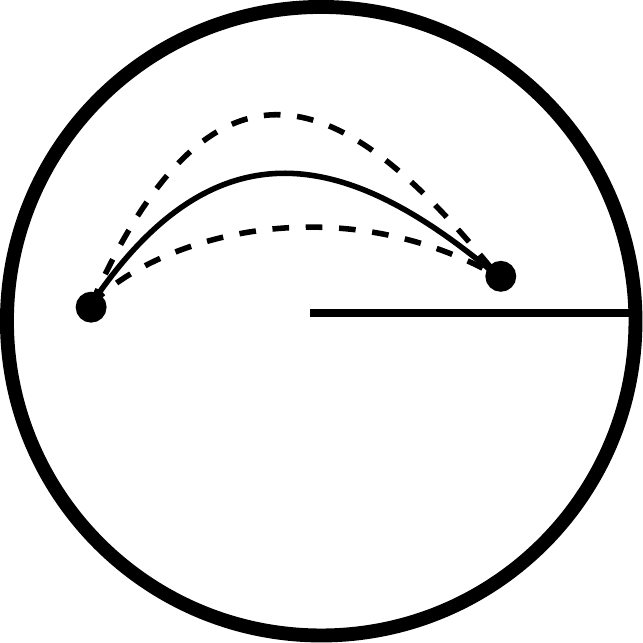}
\;\;\;
\includegraphics[width=100pt]{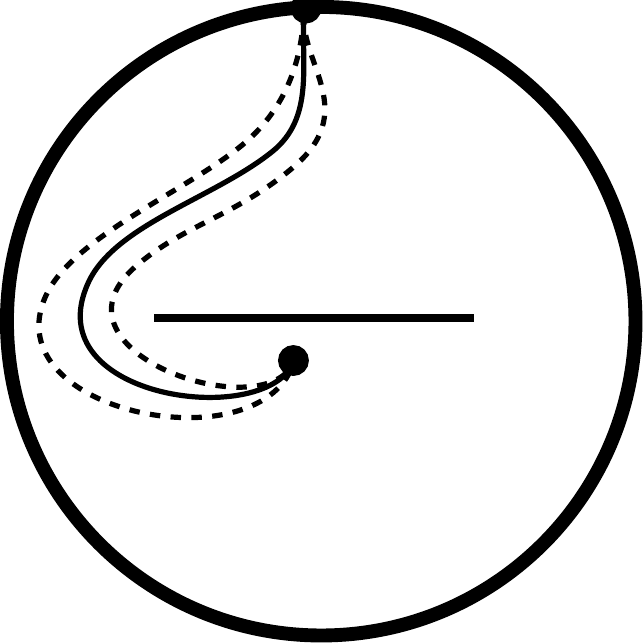}
\begin{picture}(0,0)(300,0)
\put(85,-10){a. John but not SU}
\put(195,-10){b. SU but not uniform.}
\put(100,40){$x_{0}$}
\end{picture}
\vspace{5pt}
\caption{The domain on the left is John but not SU since points on the boundary just above the slit are not accessible by short curves from points in the domain directly below the slit. The figure on the right is SU since every point on the boundary is accessible from any point in the domain by a cigar curve of bounded turning. Not every pair of points in the domain are connected in this way, since there are no short curves connecting points close to but on opposite sides of the slit, and so it is not uniform.}
\label{f:examples}
\end{figure}

Jerison and Kenig introduced NTA domains in \cite[Lemma 4.9]{JK82} since they were domains which could have very rough or non-smooth boundaries and yet their harmonic measures still had nice properties such as the doubling property: If $\Omega\subseteq \bC$ is a John domain, then harmonic measure is doubling, meaning there is $C>0$ (depending on $x\in \Omega$) so that for all balls $B$ centered on $\d\Omega$, $\omega_{\Omega}^{x}(2B)\leq C\omega_{\Omega}^{x}(B)$ (c.f. \cite[Exercise VII.13(e)]{Harmonic-Measure}). Because of this and other nice scale invariant properties, these domains have become ubiquitous in the literature on harmonic and elliptic measure (for example, \cite{Wu86,DJ90,KT06,KPT09}). We will discuss more of these properties below. 

For some of these properties, the full NTA condition is not needed. For example, Aikawa and Hirata showed that, in the case of John domains satisfying the capacity density condition, doubling is in fact equivalent to semi-uniformity.

\begin{theorem}
\label{t:AH08}
\cite{AH08} Let $\Omega\subseteq \bR^{d+1}$ be a John domain with the CDC. Then the following are equivalent:
\begin{enumerate}
\item There are $A,A_{0}>0$ so that
\begin{equation}
\label{AHdoubling}
\omega^{x}(2B)\leq A \omega^{x}(B) \mbox{ for all }x\in \Omega\backslash A_{0}B \mbox{ and $B$ centered on $\d\Omega$}.
\end{equation}
\item $\Omega$ is semi-uniform
\end{enumerate}
\end{theorem}

Another appealing property of NTA domains is the Carleson estimate.

\begin{definition}
	A corkscrew domain $\Omega\subseteq \bR^{d+1}$ has the {\bf Carleson estimate (CE)} if, for any $B$ centered on $\d\Omega$ and $u$ is a non-negative harmonic function on $2B\cap \Omega$ vanishing on $2B\cap \d\Omega$, if $y\in B\cap \Omega$ has $\dist(y,\Omega^{c})\geq \ve>0$, then 
	\[
	\sup_{B\cap \Omega} u \lec_{\ve} u(y). \]
	We say that $\Omega$ has the {\bf lesser Carleson estimate (LCE)} if there is a corkscrew point $y$ (depending on $u$) for which the above inequality holds. 
\end{definition}

It was first shown in \cite{JK82} that the CE holds in NTA domains. Later, Aikawa showed in \cite[Theorem 1.2]{Aik04} and \cite[Corollary 2]{Aik08} that, for CDC John domains, the CE condition is equivalent to $\Omega$ being uniform. 

Also proved in \cite{JK82} is the fact that for NTA domains $\Omega$, if $B$ is centered on $\d\Omega$ and $E\subseteq B$ is Borel, then there is $M>0$ so that
\begin{equation}
\label{e:jkw}
\frac{\omega_{\Omega}^{x}(E)}{\omega_{\Omega}^{x}(B)}\sim \omega^{x_{B}}_{\Omega}(E)\mbox{ for all }x\in \Omega\backslash MB. 
\end{equation}

Additionally, \eqref{e:jkw} was extended to uniform domains without exterior corkscrews. In fact, Aikawa showed that a crucial ingredient (the boundary Harnack principle, another nice scale invariant property) holds in any uniform domain \cite{Aik01}, and for uniform domains with the capacity density condition, the proof in \cite{JK82} carries over immediately. Later, Mourgoglou and Tolsa showed that \eqref{e:jkw} held in any uniform domain \cite{MT15} if we also divide the right hand side by $\omega^{x_{B}}_{\Omega}(B)$. 

These results show that the exterior corkscrew property was not necessary to acheive the same nice estimates as NTA domains, but only some nice connectivity like uniformity and some nondegeneracy in the boundary like the capacity density condition. For these reasons, uniform domains with ADR boundaries are in some sense the new NTA domains and are often studied, see \cite{HM14},\cite{HMU14}, \cite{ABHM17}, \cite{AHMNT17}, \cite{TZ17}, and \cite{HMMTZ17}. \\


Our first objective is to extend some techniques and results that are standard for NTA and uniform domains to semi-uniform domains as Aikawa and Hirata did for doubling measures. Initially, we wanted to prove Theorem \ref{t:AH08} without the John condition, but this is quite difficult. It seems that to get nice connectivity properties, we need to assume that the doubling condition \eqref{AHdoubling} also holds for balls that contain $x$. However, we can't have \eqref{AHdoubling} hold for all $x\in \Omega$, even when $\Omega$ is very nice. If $B$ is a ball on the boundary, then as $x\in \Omega$ approaches a point in $2B\cap \Omega \backslash \cnj{B}$, $\omega_{\Omega}^{x}(2B)\rightarrow 1$ whilst $\omega_{\Omega}^{x}(B)\rightarrow 0$. 

To avoid this issue, we don't have to require that $x$ remain outside some large ball $A_{0}B$, but that it stay away from the boundary of $\d\Omega$ inside that ball. 

Hence, in this paper, we will say harmonic measure is {\bf doubling} if there is a constant $A\geq 2$ and a function $C:(0,\infty)\rightarrow (1,\infty)$ so that, for any ball $B$ centered on $\d\Omega$  and $\alpha>0$, 
\begin{equation}\label{doubling}
\omega_{\Omega}^{x}(2B)\leq C(\alpha) \omega_{\Omega}^{x}(B) \mbox{ for all $x$ such that $\dist(x, AB\cap \d\Omega)\geq \alpha r_{B}$}.
\end{equation}

The work of Jerison and Kenig actually implies this stronger form of doubling, see \cite[Lemma 4.9]{JK82}. 

Our first result removes the John condition from Theorem \ref{t:AH08} using this definition of doubling.

\begin{main}\label{thmi}
Let $\Omega\subseteq \bR^{d+1}$ be a CDC domain. Then the following are equivalent:
\begin{enumerate}
\item $\omega_{\Omega}$ is doubling.
\item $\Omega$ is semi-uniform.
\end{enumerate}
\end{main}
 
Our second main result is a substitute for \eqref{e:jkw} for semi-uniform domains. 

\begin{main}
\label{thmiii}
Let $\Omega$ be a semi-uniform CDC domain, $B$ a ball centered on $\d\Omega$ with $r_{B}<\diam \d\Omega$, and $E\subseteq B\cap \d\Omega$. Then there is $M>0$ depending on the CDC and semi-uniformity constants and corkscrew points $x_{1},...,x_{n}\in B\cap \Omega$ with $n$ depending on the semi-uniformity constants so that 
\[
\min_{i=1,..,n} \omega_{\Omega}^{x_{i}}(E)
\lec \frac{\omega_{\Omega}^{x}(E)}{\omega_{\Omega}^{x}(B)}
\lec \max_{i=1,..,n}\omega_{\Omega}^{x_{i}}(E)
\mbox{ for all }x\in \Omega\backslash MB.
\]
\end{main}

For semi-uniform CDC domains, one can show that the LCE holds, and the proof is more or less the same as that in \cite{JK82} (see Lemma \ref{l:carleson} below). 
\\

We now turn to the second objective of this paper (which was also the motivation for proving the previous two theorems), which concerns the relationship between the absolute continuity properties of harmonic measure and the geometry of the domain. Very recently, it was shown in \cite{AHMMMTV16} that for $\Omega\subseteq \R^{d+1}$, if $E\subseteq \d\Omega$ (with $\cH^{d}(E)<\infty$ if $d>1$) and $\omega_{\Omega}|_{E}\ll \cH^{d}|_{E}$, then $E$ can be covered by $d$-dimensional Lipschitz graphs up to harmonic measure zero. This generalizes a result of Pommerenke who showed the same result for simply connected planar domains \cite{Pom86}. 

Being covered by Lipschitz graphs is not enough to be absolutely continuous with respect to $\cH^{d}$, however: the complement of the $\frac{2}{3}$-Cantor set as a subset of $\bR^{2}$ is one example, so some extra assumptions are needed. Bishop and Jones (generalizing work of McMillan \cite{McM69}) showed that, for simply connected planar domain, absolute continuity occurs in the subset of any rectifiable curve \cite{BJ90}. 

Higher dimensional versions of their work are false without some extra assumptions due to an example of Wu \cite{Wu86}. We showed recently with Akman and Mourgoglou that, under a lower $d$-regularity assumption on the complement of a domain $\Omega\subseteq \R^{d+1}$ (which is satisfied by many domains, including simply connected planar domains and NTA domains), $\omega_{\Omega}\ll \cH^{d}$ on any subset of a Lipschitz graph (and in fact more general surfaces) \cite{AAM16} (the techniques of which build off of previous more quantitative results, which we will describe shortly) which generalizes the work of Wu who assumed an exterior corkscrew condition \cite{Wu86}. Combining this with the work of \cite{AHMMMTV16} classifies absolute continuity for this kind of domain. 

In addition to knowing that harmonic measure and surface measure share the same null sets, one can also ask when this holds in a quantitative sense. Results of this nature typically assume some stronger properties about the surface measure on the boundary. The first example (which forms the foundation of all subsequent results establishing absolute continuity in higher dimensions) is the result of Dahlberg \cite{Dah77} that harmonic measure for a Lipschitz domain is an $A_{2}$-weight. 

Recall that a set $E$ is {\bf $d$-Ahlfors-David regular (ADR)}, or just {\bf $d$-regular}, if there is $C>0$ such that 
\[
C^{-1} r^{d} \leq \cH^{d}(B(x,r)\cap E)\leq Cr^{d} \mbox{ for all $x\in E$ and $0<r<\diam E$}.\]

We will say that harmonic measure is $A_{\infty}$ if , for all $\alpha>0$, $\omega_{\Omega}^{x}\in A_{\infty}(\d\Omega\cap B,\cH^{d})$ for any ball $B$ centered on $\d\Omega$ and $x\in \Omega$ with $\dist(x,AB)>\alpha r_{B}$. That is, for all $\delta>0$ there is $\ve>0$ (also depending on $\alpha$) so that if $E\subseteq \d\Omega\cap B$ and 
\[
\cH^{d}(E)<\ve \cH^{d}(B\cap \d\Omega),\]
then 
\[
\omega_{\Omega}^{x}(E)<\delta \omega_{\Omega}^{x}(B).\]

Seeking out this form of absolute continuity has applications for PDEs: in \cite{HL16}, for example, Hofmann and Le showed that BMO solvability of the Dirichlet problem for the Laplacian is implied by the $A_{\infty}$-property (in fact, it is implied by the weak $A_{\infty}$ property). 

We say a domain is a {\bf chord-arc domain (CAD)} if it is NTA with Ahlfors regular boundary. In the plane, this is equivalent to the boundary being a chord-arc (or bi-Lipschitz) curve. Lavrentiev showed in \cite{Lav36} that, for chord-arc domains in the plane, harmonic measure is in $A_{\infty}$, and in fact, for Jordan domains with Ahlfors regular boundaries, the converse holds as well (for modern treatments of both these facts, see \cite[Section VII.4]{Harmonic-Measure}). Independently, David and Jerison \cite{DJ90} and Semmes \cite{Sem90} proved Lavrentiev's theorem for higher dimensions. The common thread to both proofs is to reduce things to Dahlberg's original result by approximating the domain from within by Lipschitz subdomains. In particular, in \cite{DJ90} the authors first prove that a CAD has {\bf big pieces of Lipschitz subdomains (BPLS)}: for every ball $B=B(x.r)$ $x\in \d\Omega$ and $0<r<\diam \Omega$, there is a Lipschitz domain $\Omega_{B}\subseteq B\cap \Omega$ so that $\cH^{d}(\d\Omega_{B}\cap \d\Omega)\geq cr^{d}$. Dahlberg's result shows that harmonic measure is an $A_{\infty}$-weight, and then using the maximum principle one can show that the $A_{\infty}$ property for the original domain is inherited from these subdomains.

To date, these are the most general domains for which the $A_{\infty}$ property (as we have defined it) has been proven to hold, and there is yet no result that says exactly for which domains it holds. The only exception are when assuming the domain is uniform with ADR boundary (see \cite{HM14,HMU14}), but even in this setting, the $A_{\infty}$ property is actually equivalent to the domain being a CAD \cite{AHMNT17}.  The most general kind of domain for which the $A_{\infty}$ property holds that follows immediately from results in the literature (although isn't stated anywhere) are semi-uniform domains with BPLS: Bennewitz and Lewis showed  in \cite{BL04} that harmonic measure satisfies a weak-reverse H\"older inequality in corkscrew domains with BPLS; in semi-uniform domains, the corkscrew property is immediate, and because harmonic measure is doubling, harmonic measure actually satisfies the usual reverse H\"older inequality and hence our $A_{\infty}$ condition by classical results (see \cite[Chapter 5]{Ste93}). 

There are some necessary conditions our domain must satisfy for harmonic measure to be $A_{\infty}$. Firstly, Hofmann and Martell showed that the boundary is {\bf uniformly rectifiable (UR)} \cite{HM15}\footnote{They actually show that the so-called weak-$A_{\infty}$ property implies UR, although we will not discuss this class of measures.} : $\d\Omega$ is $d$-regular and there are $L,c>0$ so that, for each ball $B(x,r)$ centered on $\d\Omega$ with $0<r<\d\Omega$ there is an $L$-Lipschitz map $f:\bR^{d}\rightarrow \bR^{d+1}$ so that 
\[
\cH^{d}(B(x,r)\cap \d\Omega\cap f(\R^{d}))\geq cr^{d}.\] 
This paper is an Arxiv preprint, although later they extended this result to $p$-harmonic measures in a paper with Le and Nystr\"om \cite{HLMN17}. Mourgoglou and Tolsa also developed a local result that works when harmonic measure is not doubling \cite{MT15}.  See also \cite{HMT16} and \cite{HMMTZ17} for the elliptic setting. Secondly, since $A_{\infty}$-weights are doubling, Theorem \ref{thmi}  implies that such domain must also be semi-uniform. Our third result confirms that these conditions are also sufficient, thus classifying the $A_{\infty}$ property for harmonic measure. 

\begin{main}\label{thmii}
Let $\Omega\subseteq \bR^{d+1}$ be a domain with $d$-regular boundary. Then the following are equivalent:
\begin{enumerate}
\item $\Omega$ is a semi-uniform domain with ADR and UR boundary. 
\item $\Omega$ is a semi-uniform domain with ADR boundary and very big pieces of chord-arc subdomains (VBPCAS): for every ball $B$ centered on $\d\Omega$ and $\ve>0$, there is $\Omega_{B}\subseteq B\cap \Omega$ a CAD (with constants depending on the semi-uniformity, Ahlfors regularity, and on $\ve$) so that 
\[
\cH^{d}(\d\Omega\cap B\backslash \d\Omega_{B})<\ve \cH^{d}(\d\Omega\cap B).
\]
\item Harmonic measure is $A_{\infty}$. 
\end{enumerate}
\end{main}

Recall that, combining the works of \cite{HMU14} and\cite{AHMNT17}, for uniform domains with ADR boundary, the $A_{\infty}$ property is equivalent to the boundary being UR and equivalent to the domain being CAD. This equivalence doesn't hold for semi-uniform domains, since the complement of a line segment clearly satisfies the conditions of Theorem \ref{thmii} without being CAD.

[Addendum (February 21, 2018): Shortly after posting this paper to Arxiv, Hofmann and Martell posted another paper \cite{HM17} where they showed that the weak $A_{\infty}$ condition is implied by a so-called ``weak local John condition." This means that, for every $x\in \Omega$ there is a set $F\subseteq \d\Omega\cap B(x,2\dist(x,\Omega^{c}))$ so that, for all $\xi\in F$, there is a path $\gamma\subseteq \Omega$ from $x$ to $\xi$ so that $\dist(z,\d\Omega)\geq c |z-\xi|$ for all $z\in \gamma$. Of course, this is weaker than being semi-uniform, and also, if we assume semi-uniformity, then this result combined with Theorem \ref{thmiii} gives an alternate proof that (1) implies (3) in Theorem \ref{thmii}. Indeed, since this implies weak $A_{\infty}$, we can obtain the local $A_{\infty}$ property using doubling and then our global $A_{\infty}$ condition by using Theorem \ref{thmiii} as in the proof of Theorem \ref{thmii} in Section \ref{s:proof-of-theorem}. However, the techniques and approach of Hofmann and Martell in their paper are quite different. They use the method of ``extrapolation" of Carleson measures to estimate harmonic measure directly, whereas we model our argument on that of \cite{DJ90} by building nice chord-arc subdomains that carve out a large piece of the boundary (that is, the implication that (1) implies (2) in Theorem \ref{thmii}, which is proven in Lemma \ref{l:CAD} below) to prove a local $A_{\infty}$ property, and then we use Theorem \ref{thmiii} to get the global $A_{\infty}$ property.]

In light of Bennewitz and Lewis' result mentioned earlier, however, it would be natural to ask if condition (1) was also equivalent to being semi-uniform with ADR boundary and BPLS. This is certainly true for CADs, as shown by David and Jerison in \cite{DJ90} (and this was crucial for their proof), so in the uniform setting, the $A_{\infty}$ property implies BPLS. However, there are examples of semi-uniform domains with ADR and UR boundary that do not have BPLS. One example is the complement of Hrycak's example, a well-known (and unpublished) set constructed by Hrycak, often cited in the literature on uniform rectifiability to show that not all UR sets have big pieces of Lipschitz graphs. In the appendix, we show the following.

\begin{prop}\label{p:main}
If $E\subseteq \R^{2}$ is Hrycak's example, then $E^{c}$ is a semi-uniform domain with UR boundary and does not have BPLS.
\end{prop}

As as mentioned before, the implication (3) implies (1) in Theorem \ref{thmii} follows by Theorem \ref{thmi} and \cite{HM15}, so all we will show in this paper is the implication that (1) implies (2) and (2) implies (3). The proof of (1) implies (2) requires building chord-arc subdomains of $\Omega$ that carve out as much of the boundary as we please, and we know that harmonic measure is $A_{\infty}$ in this domain by \cite{DJ90}. We then use Theorem \ref{thmiii} and the maximum principle to prove $A_{\infty}$ for our original measure. The work of Bortz and Hofmann \cite{BH17} comes close to what we need by building a {\it union} of (possibly disjoint) chord-arc domains, and in essence what we do is show that these chord arc domains can be connected into one single CAD, although our construction in the end is quite different and uses some additional techniques in order to prove semi-uniformity. 
\\


We would like to thank Mihalis Mourgoglou and Xavier Tolsa for their helpful discussions and comments on the paper, Hiroaki Aikawa for answering our questions about semi-uniform domains, Alan Chang for carefully proofreading and correcting the appendix, and also the referees for their suggested corrections that greatly improved the paper.

\section{Preliminaries}

\subsection{Notation}

We will write $a\lesssim b $ if there is a constant $C>0$ so that $a\leq C b$ and $a\lesssim_{t} b$ if the constant depends on the parameter $t$. As usual we write $a\sim b$ and $a\sim_{t} b$ to mean $a\lesssim b \lesssim a$ and 
$a\lesssim_{t} b \lesssim_{t} a$ respectively. We will assume all implied constants depend on $d$ and hence write $\sim$ instead of $\sim_{d}$.

Whenever $A,B\subset\mathbb{R}^{d+1}$ we define
\[
\mbox{dist}(A,B)=\inf\{|x-y|;\, x\in A, \, y\in B\}, \, \mbox{and}\, \, \mbox{dist}(x,A)=\mbox{dist}(\{x\}, A). 
\]
Let $\diam A$ denote the diameter of $A$ defined as
\[
\diam A=\sup\{|x-y|;\, x,y\in A\}.
\]

For a domain $\Omega$ and $x\in \Omega$, we will write
\[
\delta_{\Omega}(x) =\dist(x,\d\Omega).
\]

We let $B(x,r)$ denote the open ball centered at $x$ of radius $r$. For a ball $B$, we will denote its radius by $r_{B}$. 

\subsection{Harnack Chains}

\begin{definition}
	Let $\Omega\subseteq \bR^{d+1}$. A {\it Harnack chain} is a (finite or infinite) sequence of balls $\{B_{i}\}_{i\in [a,b]}$ where $[a,b]$ denote the integers between $a$ and $b$ and $a$ can be $-\infty$ and $b$ can be $+\infty$, such that for all $i$,
	\begin{enumerate}
		\item $B_{i}\cap B_{i+1}\neq\emptyset$ if $a\leq i< b$, 
		\item $2B_{i}\subseteq \Omega$, and 
		\item $r_{B_{i}} \sim \dist(B_{i},\d\Omega)$. 
	\end{enumerate}
	The {\it length} of a Harnack chain is just the number of balls in the Harnack chain. 
\end{definition}

Note that if $B_{i}$ is a Harnack chain, 
\[
r_{B_{i}}\sim r_{B_{i+1}}.
\]

%
%

To verify that a domain is either uniform or semi-uniform, it will be more convenient to work with equivalent definitions in terms of Harnack chains. 

\begin{theorem}
\label{l:ahmnt}
\cite[Theorem 2.15]{AHMNT17}
A domain $\Omega$ is uniform if and only if it has interior corkscrews and there is a non-decreasing function $N:[1,\infty)\rightarrow [1,\infty)$ so that for all $x,y\in \Omega$, there is a Harnack chain from $x$ to $y$ in $\Omega$ of length $N(|x-y|/\min\{\delta_{\Omega}(x),\delta_{\Omega}(y)\})$. 
\end{theorem}

One can prove a similar Harnack chain version of semi-uniformity.

\begin{theorem}\label{SU}
A domain $\Omega\subseteq \bR^{d+1}$ is semi-uniform if and only if  it has interior corkscrews and there is $c>0$ and a non-decreasing function $N:[1,\infty)\rightarrow [1,\infty)$ so that, for all $x\in \Omega$, $\xi\in \d\Omega$, and $0<r<\diam \d\Omega$, there is a Harnack chain from $x$ to a $c$-corkscrew point $y\in \Omega\cap B(\xi,r)$ of length $N(|x-y|/\min\{\delta_{\Omega}(x),r\})$. 
\end{theorem}

\begin{remark}
\label{r:referee}
The condition that $r<\diam \d\Omega$ is important. Note that if $\Omega = \bR^{d+1}\backslash \d \bB\cup B(e_{d+1},\ve)$ where $e_{d+1}$ is the $(d+1)$st standard basis vector and $\ve$ is small, this is an unbounded semi-uniform domain. However, points in $\bB$ are not well connected to corkscrew points outside $\bB$ because of the small $\ve$-hole, that is, the property stated in the previous theorem does not hold if we allow $r\gg \diam \d\Omega$. So in particular, the statement implies that a corkscrew ball in a ball $B$ centered on the boundary can be connected down to a smaller corkscrew ball in $M^{-1}B$ with length depending on $M$ {\it for any $M\geq 1$}, but it can only be connected up to a larger corkscrew ball in $MB$ with length depending on $M$ {\it so long as} $Mr_{B}<\diam \d\Omega$. 
\end{remark}

\begin{proof}
We only sketch the details. For the forward direction, one takes a cigar curve $\gamma$ from $\xi$ to $x$ and then one can show that a Besicovitch subcover of $\{B(z,\delta_{\Omega}(z)/2):z\in \gamma\}$ gives the desired Harnack chain. For the reverse direction, the proof of this is similar to the proof of \cite[Theorem 2.15]{AHMNT17}, but we will outline the initial steps.

Assume we have a function $N$ satisfying the properties in the theorem. Let $x\in \Omega$ and $\xi\in \d\Omega$ (and note that the ball $B(\xi,r)$ could be much larger than $B(\xi,|x-\xi|)$).

\begin{enumerate}
\item If $\delta_{\Omega}(x)\geq |\xi-x|/8$, let $B'=B(\xi,2|x-\xi|)$. Then by iterating, we can find fir each $i\geq 0$ corkscrew points $x_{i}\in 2^{-i}B'$ (with $x_{0}=x$) and Harnack chain $B_{1}^{i},...,B_{n_i}^{i}$ from $x_{i}$ to $x_{i+1}$. If we connect the centers of all the Harnack chains in order by line segments (so we connect the centers of $B_{j}^{i}$ to $B_{j+1}^{i}$ and $B_{n_{i}}^{i}$ to $B_{1}^{i+1}$), one can show this is a cigar curve of bounded turning as in the proof of \cite[Theorem 2.15]{AHMNT17}. 
\item If $\delta_{\Omega}(x)<|\xi-x|/8$, let $\zeta\in \d\Omega$ be closest to $x$ and $B''=B(\zeta,2|x-\xi|)$. Let $n$ be the largest integer for which $2^{n}r_{B''}<r_{B'}/4$. Since 
\[
2r_{B''}=4|\zeta-x|<|\xi-x|/2=r_{B'}/4,\]
we know $n> 0$. Then

\[
r_{2^{n}B''}
=2^{n} r_{B''}
<\frac{r_{B'}}{4}
=\frac{|\xi-x|}{2}
\leq \frac{|\xi-\zeta|+|\zeta-x|}{2}
\leq \frac{\diam \d\Omega}{2} + \frac{r_{B''}}{4}
\]
hence, since $n>0$,
\[
2^{n}r_{B''}<\frac{2^{n}}{2^{n}-1/4}\frac{\diam \d\Omega}{2}< \diam \d\Omega.
\]

Thus, we can apply the condition of the theorem to get that, for $0\leq i\leq n$, there are corkscrew points $y_{i}\in 2^{i}B''$ and a bounded Harnack chain $\tilde{B}_{1}^{i},...,\tilde{B}_{m_{i}}^{i}$ from $y_{i}$ to $y_{i+1}$. If we connect the centers of these balls in order we obtain a curve $\gamma_1$. Note that 
\[
\delta_{\Omega}(y_{n})\gec 2^{n}r_{B''}\sim r_{B'}\sim |x-\xi|
\]
and so just as in the previous case, we can find a curve $\gamma_{2}$ connecting the centers of an infinite Harnack chain from $y_{n}$ to $\xi$. The union of these two curves $\gamma$ can be shown as in the proof of \cite[Theorem 2.15]{AHMNT17} to be cigar curves of bounded turning. 
\end{enumerate}

\end{proof}

\subsection{Background on Harmonic Measure}

For background on harmonic measure and Green's function, we refer the reader to \cite{AG}. 
\begin{definition}
For $K\subset \d\Omega$, we say that $\Omega$ has the {\it capacity density condition (CDC) in $K$} if $ \textup{cap}({B}(x,r) \cap \Omega^c, B(x,2r)) \gtrsim r^{d-1}$, for every $x \in K$ and $r<\diam K$, and that $\Omega$ has the {\it capacity density condition} if it has the CDC in $K=\d\Omega$.  Here, cap$(\cdot, \cdot)$ stands for the variational $2$--capacity of the condenser $(\cdot, \cdot)$ (see \cite[p. 27]{HKM} for the definition).
\end{definition}

\begin{remark}
\label{r:cdc}
This is the traditional definition of CDC, but it also has a geometric formulation. By the main result of \cite{Leh08} (also see \cite[Equation (8)]{Leh08}), $\Omega\subseteq \R^{d+1}$ satisfies the CDC (or is {\it uniformly $2$-fat} in that paper's argot) if there are $c>0$ and $s>d-1$ so that 
\[
\cH^{s}_{\infty}(B(\xi,r)\backslash \Omega)\geq cr^{s} \mbox{ for all }x\in \Omega\mbox{ and }r>0.\]
Below, when we talk about the {\it CDC constants}, we will in fact refer to the constants $s$ and $c$ here. 
\end{remark}

\begin{lemma}[{\cite[Lemma 11.21]{HKM}}]\label{l:bourgain}
Let $\Omega\subset \bR^{d+1}$ be any domain satisfying the CDC condition,  $B$ a ball centered on $\d\Omega$ so that $\Omega\backslash 2B\neq\emptyset$. Then 
\begin{equation}\label{e:bourgain}
\omega_{\Omega}^{x}(2B)\geq c >0 \;\; \mbox{ for all }x\in \Omega\cap B.
\end{equation}
where $c$ depends on $d$ and the constant in the CDC.
\end{lemma}

Using the previous lemma and iterating, it is possible to obtain the following lemma.

\begin{lemma}\label{l:holder}
Let $\Omega\subseteq \bR^{d+1}$ be a domain with the CDC, $\xi\in \d\Omega$ and $0<r<\diam \d\Omega/2$. Suppose $u$ is a non-negative function that is harmonic in $B(\xi,r)\cap \Omega$ and vanishes continuously on $\d\Omega\cap B(\xi,r)$. Then
\begin{equation}\label{e:holder}
u(x) \lec  \ps{\sup_{y\in B(\xi,r)\cap \Omega} u} \ps{\frac{|x-\xi|}{r}}^{\alpha}
\end{equation}
where $\alpha>0$ depends on the CDC constant and $d$. 
\end{lemma}

%


There are two key facts we will use about Green's function.

\begin{lemma}
\cite[Lemma 1]{Aik08}
For $x\in \Omega\subseteq \bR^{d+1}$ and $\phi\in C_{c}^{\infty}(\bR^{d+1})$, 
\begin{equation}
\label{e:ibp}
\int \phi\omega_{\Omega}^{x} = \int_{\Omega} \triangle \phi(y) G_{\Omega}(x,y)dy.
\end{equation}
\end{lemma}

\begin{lemma}
Let $\Omega\subset\bR^{d+1}$ be a CDC domain. Let $B$ be a ball centered on $\d\Omega$ and $0<r_{B}<\diam \d\Omega $. Then,
 \begin{equation}\label{w>G}
 \omega^{x}(4B)\gtrsim r_{B}^{d-1}\, G_{\Omega}(x,y)\quad\mbox{
 for all $x\in \Omega\backslash  2B$ and $y\in B\cap\Omega$,}
 \end{equation}
\end{lemma}

This follows quickly from the maximum principle, Lemma \ref{l:bourgain}, and the fact that, for $x\in \d 2B\cap \Omega$ and $y\in B$, $r_{B}^{d-1}G_{\Omega}(x,y)\lec 1$. For proofs, see \cite[Lemma 3.5]{AH08} or \cite[Lemma 3.3]{AHMMMTV16}.

Some of the proofs below will use compactness arguments via the following lemma from \cite{AMT17}. 

\begin{lemma}\label{limlem} \cite[Lemma 2.9]{AMT17}
Let $\Omega_{j}\subseteq \bR^{d+1}$ be a sequence of CDC domains with the same CDC constants (as in Remark \ref{r:cdc}) such that $0\in \d\Omega_{j}$, $\inf\diam \d\Omega_{j}>0$, and there is a ball $B(x_{0},r)\subseteq \Omega_{j}$ for all $j$. Then there is a connected open set $\Omega_{\infty}^{x_{0}}$ containing $B(x_{0},r)$ so that, after passing to a subsequence,
\begin{enumerate}
\item $G_{\Omega_{j}}(x_{0},\cdot)$ converges uniformly to $G_{\Omega_{\infty}^{x_{0}}}(x_{0},\cdot)$ on compact subsets of $\{x_{0}\}^{c}$,
\item   $\omega_{\Omega_{j}}^{x_{0}}\warrow \omega_{\Omega_{\infty}^{x_{0}}}^{x_{0}}$, and
\item  $\Omega_{\infty}^{x_{0}}$ has the CDC with the same constants. 
\end{enumerate}
\end{lemma}

We'll need an additional two lemmas building off of this one.

\begin{lemma}
With the assumptions of Lemma \ref{limlem}, if $x\in \Omega_{\infty}^{x_{0}}$, then we may pass to a further subsequence so that the same conclusions hold with $x$ in place of $x_{0}$ and $\Omega_{\infty}^{x_{0}}=\Omega_{\infty}^{x}$. In particular, $\omega_{\Omega_{\infty}^{x_{0}}}^{x_{0}}=\omega_{\Omega_{\infty}^{x_{0}}}^{x}$. 
\end{lemma}

\begin{proof}
When passing to the subsequence in Lemma \ref{limlem}, we can pass to another subsequence so that $G_{\Omega_{j}}(x,\cdot)$ converges on compact subsets of $\{x\}^{c}$ to $G_{\Omega_{\infty}^{x}}(x,\cdot)$. For $y \in \Omega_{\infty}^{x_{0}} \backslash \{x_{0},x\}$, a small ball around $y$ is contained in $\Omega_{j}$ for all $j$ large (since $G_{\Omega_{\infty}^{x_{0}}}(x_{0},\cdot)$ is bounded away from zero on this ball, hence so is $G_{\Omega_{j}}(x_{0},\cdot)$ for large $j$).
Let $\gamma\subseteq \Omega_{\infty}^{x_{0}}\backslash\{y\}$ be a curve from $x$ to $x_{0}$. Then $\gamma \subseteq \Omega_{j}$ for sufficiently large $j$ (this is because $G_{\Omega_{\infty}^{x_{0}}}(x_{0},\cdot)\geq c>0$ on $\gamma$ for some $c>0$, and thus $G_{\Omega_{j}}(x_{0},\cdot)\geq c/2>0$ for $j$ large). Thus, by Harnack's inequality used on a chain of balls along $\gamma$ whose doubles don't contain $y$,
\[
G_{\Omega_{j}}(x_{0},y)\sim_{y} G_{\Omega_{j}}(x,y)
\]
In particular,
\[
G_{\Omega_{\infty}^{x_{0}}}(x_{0},y)\sim_{y} G_{\Omega_{\infty}^{x}}(x,y).
\]
and so
\begin{align*}
\Omega_{\infty}^{x_{0}}\backslash\{x,x_{0}\} 
& =\{y\neq x,x_{0}: G_{\Omega_{\infty}}^{x_{0}}(x_{0},y)>0\}\\
& =\{y\neq x,x_{0}: G_{\Omega_{\infty}}^{x}(x,y)>0\}\\
& =\Omega_{\infty}^{x}\backslash\{x,x_{0}\}.
\end{align*}

Thus, adding back $x$ and $x_{0}$, we get $\Omega_{\infty}^{x_{0}}=\Omega_{\infty}^{x}\cup \{x_{0}\}$. Since the former set is open, so must the latter set, and this can only be if $x_{0}\in \Omega_{\infty}^{x}$. Indeed, if $x_{0}\not\in \Omega_{\infty}^{x}$, then since $\Omega_{\infty}^{x}\cup \{x_{0}\}$ is open, there is $\ve>0$ so that $B(x_{0},\ve)\subseteq \Omega_{\infty}^{x}\cup \{x_{0}\}$, and so $B(x_{0},\ve)\backslash \{x_{0}\}\subseteq \Omega_{\infty}^{x}$. Thus, $x_{0}$ is an isolated point of $\d\Omega_{\infty}^{x}$, but this is impossible since domains with the CDC have no isolated points in their boundary. This proves the lemma.
\end{proof}

\begin{lemma}
\label{l:doublinglimit}
Under the conditions of Lemma \ref{limlem}, if $\omega_{\Omega_{j}}^{x_{0}}$ satisfies \eqref{doubling}, then so does $\omega_{\Omega_{\infty}}^{x_{0}}$ with the same constants.
\end{lemma}

\begin{proof}
Let $B$ be centered on $\d\Omega_{\infty}^{x_{0}}$, $\alpha>0$, and $x\in \Omega_{\infty}^{x_{0}}$ such that $\dist(x,AB\cap \d\Omega_{\infty}^{x_{0}})\geq \alpha r_{B}$. Then there is $\xi_{j}\in \d\Omega_{j}$ so that $\xi_{j}\rightarrow x_{B}$. For every $0<r<s<t<1$, if $j$ large enough and $B_{j}=B(\xi_{j},r_{B})$, then $2rB\subseteq 2sB_{j}$, $sB_{j}\subseteq tB$, and 
\[
\dist(x,AsB_{j}\cap \d\Omega_{j})
\geq  \alpha r_{sB_{j}}=\alpha sr_{B}.\]
Otherwise, if for infinitely many $j$ we could find $\zeta_{j}\in B(x,\alpha sr_{B})\cap AsB_{j}$, then by passing to a subsequence, they converge to a point $\zeta\in B(x,\alpha r_{B})$. Since the $\omega_{\Omega_{j}}^{x}$ are uniformly doubling and there is a small ball containing $x_0$ that is contained in $\Omega_{j}$ for all $j$ large, $\omega_{\Omega_{j}}^{x_0}(B(\xi,\ve))\gec_{\ve,x} 1$ for all large $j$ and $\ve>0$, hence $\omega_{\Omega_{\infty}^{x_{0}}}^{x_0}(B(\xi,\ve))\gec_{\ve,x} 1$ as well, so $\zeta\in \supp \omega_{\Omega_{\infty}}^{x_0}=\d\Omega_{\infty}^{x_{0}}$, but then $\dist(x,AB\cap \d\Omega_{\infty})<\alpha r_{B}$, a contradiction. 

Thus, recalling the previous lemma,
\begin{align*}
\omega_{\Omega_{\infty}^{x_{0}}}^{x}(2rB)
& \leq \liminf_{j} \omega_{\Omega_{j}}^{x}(2rB)\\
& \leq \liminf_{j} \omega_{\Omega_{j}}^{x}(2sB_j)\\
& \stackrel{\eqref{doubling}}{\leq} \liminf_{j}  C(\alpha)\omega_{\Omega_{j}}^{x}(sB_j)\\
& \leq C(\alpha) \omega_{\Omega_{\infty}^{x_{0}}}^{x}(\cnj{tB})
.\end{align*}
Letting $t\uparrow 1$ and $r\uparrow 1$ now gives \eqref{doubling} when $\Omega=\Omega_{\infty}^{x_{0}}$.

\end{proof}

\section{Proof of \TheoremI}

This section is dedicated to the proof of Theorem \ref{thmi}. \\

The reverse implication follows using Theorem \ref{t:AH08} and semi-uniformity. Suppose $\Omega$ is semi-uniform, then Theorem \ref{t:AH08} implies \eqref{AHdoubling} for some constant constants $A$ and $A_0$, which is a priori weaker than \eqref{doubling}. Let $\alpha>0$, and let $x$ be such that 
\begin{equation}
\label{e:xA_0B}
\dist(x, A_0B\cap \d\Omega)\geq \alpha r_{B}.
\end{equation}
We split into three cases:
\begin{enumerate}
\item If $x\in \Omega\backslash A_0 B$, then we have $\omega^{x}_{\Omega}(2B)\leq A\omega^{x}_{\Omega}(B)$ immediately by \eqref{AHdoubling}.
\item Now suppose $x\in A_0 B$ and $\Omega\backslash 2A_{0}B\neq\emptyset$, then by semi-uniformity and \eqref{e:xA_0B}, we can find a Harnack chain from $x$ to a point $y\in \Omega\backslash A_{0}B$ (with length depending on $\alpha$ and $A_{0}$, which depends on the semi-uniformity constant), and so 
\[
\omega_{\Omega}^{x}(2B)\sim_{\alpha,A_{0}} \omega_{\Omega}^{y}(2B)
\stackrel{\eqref{AHdoubling}}{\leq} A \omega_{\Omega}^{y}(B)
\sim_{\alpha,A_{0}} A\omega_{\Omega}^{x}(B).
\]
\item If $x\in \Omega\backslash A_0 B$ and $\Omega\subseteq 2A_{0}B$, then by semi-uniformity, there is a Harnack chain from $x$ to a corkscrew point $y\in \frac{1}{2}B$, again with length depending on $\alpha$ and $A_{0}$. Thus,
\[
\omega_{\Omega}^{x}(B) 
\sim_{\alpha} \omega_{\Omega}^{y}(B) 
\stackrel{\eqref{e:bourgain}}{\gec} 1
\geq \omega_{\Omega}^{x}(2B) .
\]
\end{enumerate}
For the rest of this section, we will focus on showing that if $\omega_{\Omega}$ is doubling in the sense of \eqref{doubling}, then $\Omega$ is semi-uniform.

\begin{lemma}\label{corkscrew}
If \eqref{doubling} holds, then $\Omega$ has interior $c_{1}$-corkscrews with $c_{1}\in (0,1/8)$ depending on the CDC and doubling constants. 
\end{lemma}
The constant $c_{1}$ may be larger than $1/8$, but if a domain has $c_1$ corkscrews, then it has $c$-corkscrews for $c<c_1$, so the conclusion is still true. It will just be convenient to assume $c_{1}<1/8$ for later on.

\begin{proof}
Suppose there is a sequence of domains $\Omega_{j}$ for which \eqref{doubling} is satisfied, and a sequence of balls $B_{j}$ centred on $\d\Omega$ with $0<r_{B_{j}} < \diam \Omega_{j}/2A$ so that 
\[
\sup_{x\in B_{j}\cap \Omega} \dist(x,\d\Omega) <\frac{r_{B_{j}}}{j}.\]
Without loss of generality, by scaling and translating our domain, we can assume $0\in \d\Omega_{j}$ and $B_{j}= \bB$, so $r_{B_{j}}=1$. Then $\diam \d\Omega_{j}>2A$. By rotating we may assume that there is a fixed point $y\in \Omega_{j}\backslash 2A\bB$ for all $j$. 

By Lemmas \ref{l:bourgain} and \ref{l:holder}, if we define
\[
g_{j}(x) =\isif{ \omega_{\Omega_{j}}^{y}(\bB)^{-1} G_{\Omega_{j}}(x,y) & x\in \Omega \\ 0 & x\not\in \Omega}\]
then 
\[
|g_{j}(x)|\lec j^{-\alpha} \mbox{ for all }x\in \bB\cap \Omega\]
and so $g_{j}\rightarrow 0$ uniformly in $\bB$. However, if $\phi\in C_{c}(\bB)$ is equal to $1$ on $\frac{1}{2}\bB$, then by \eqref{doubling} (since $\dist(y,A\bB\cap \d\Omega)\geq A$) and \eqref{e:ibp}, 
\[
1\lec\frac{\omega_{\Omega_{j}}^{y}(\frac{1}{2}\bB)}{\omega_{\Omega_{j}}^{y}(\bB)} \leq \int \phi  \frac{d\omega_{\Omega_{j}}^{y}}{\omega_{\Omega_{j}}^{y}(\bB)} = \int g_{j}\triangle \phi d x \rightarrow 0\]
which is a contradiction. 

\end{proof}

\begin{lemma}\label{down}
Let $\Omega\subseteq \bR^{d+1}$ is a CDC domain, $c>0$ and assume \eqref{doubling} holds. Let $c_{1}$ be as in the previous lemma. For any ball $B$ centered on $\d\Omega$ and $B'\subseteq B\cap \Omega$ a $c$-corkscrew ball, there is a Harnack chain of length $N$ (depending on $c$, the CDC constants, and the doubling constants) whose first ball is $B'$ and whose last ball is a $c_{1}/2$-corkscrew ball for $\frac{1}{2}B$.
\end{lemma}

\begin{proof}
Without loss of generality, we can assume $c<\frac{1}{4}$ and $x_{B}=0$.

Suppose $\frac{4}{c}\diam \d\Omega < r_{B}$, then $\d\Omega\subseteq \frac{c}{4} B\subseteq \frac{1}{4}B$, and since $B'$ is a corkscrew ball, $2B'\subseteq B\cap \Omega$, hence $\frac{1}{2}B'\subseteq (\frac{c}{2}B)^{c}$. We can then connect $B'$ by a short Harnack chain in $B\backslash \frac{c}{4}B$ to the center of a ball $B''$ of radius $\frac{1}{16}r_{B}$ with $2B''\subseteq  \frac{1}{2} B\backslash \frac{1}{4}B$, so $B''$ is a $\frac{1}{8}$-corkscrew ball for $\frac{1}{2} B$. Since $c_{1}\leq \frac{1}{8}$, $B''$ is also a $c_{1}$-corkscrew for $\frac{1}{2}B$. Now we must prove the lemma in the case that $\frac{4}{c}\diam \d\Omega \geq  r_{B}$

Suppose there were domains $\Omega_{j}$ and balls $B_{j}$ centered on $\d\Omega$ with $0<r_{B_{j}}\leq \frac{4}{c}\diam \d\Omega_{j}$ whose harmonic measures were doubling  in the sense of \eqref{doubling} (with the same constants) and $c$-corkscrew balls $B_{j}'\subseteq B_{j}\cap \Omega_{j}$ for which the shortest Harnack chain whose first ball is $B_{j}'$ and whose last ball is a $c_{1}/2$-corkscrew ball contained in $\frac{1}{2}B$ has length at least $j$ (since $\Omega_{j}$ is connected and has the $c_{1}$-corkscrew property, this is well defined). Without loss of generality, we may assume $B_{j}=\bB$. By passing to a subsequence if necessary, we can find a ball $B(x_{0},c-1/j)\subseteq B_{j}'$ for all $j$. Again, we can pass to a subsequence so that the conclusions of Lemma \ref{limlem} hold. In particular, if $\Omega_{\infty}^{x_{0}}$ is from the lemma, then it is doubling by Lemma \ref{l:doublinglimit} with the same constants and so it also has the $c_{1}$-corkscrew property.

Using the Harnack principle, the doubling property, and Lemma \ref{e:bourgain}, we have that for all $0<r<1$, 
\[
\omega_{\infty}^{x_{0}}(r\bB)
\geq \limsup_{j\rightarrow\infty} \omega_{\Omega_{j}}^{x_{0}}(\frac{r}{2} \bB)\gec \omega_{\Omega_{j}}^{x_{B_{j}'}}(\frac{r}{2} \bB)\gec_{r} \omega_{\Omega_{j}}^{x_{B_{j}'}}(\bB)\gec 1.
\]
Thus, $\omega_{\infty}^{x_{0}}(r\bB)>0$ for all $0<r<1$, which implies
\[
0\in \supp \omega_{\infty}^{x_{0}} =\d\Omega_{\infty}^{x_{0}}.\] 
Since $\Omega_{\infty}^{x_{0}}$ is a connected $c_{1}$-corkscrew domain, there is a finite Harnack chain from a $c_{1}$-corkscrew ball $B'$ for $\frac{1}{2}\bB$ in $\Omega_{\infty}^{x_{0}}$ to $x_0$ contained in $\Omega_{\infty}^{c_{0}}$. Hence, there is a ball $B''\subseteq B'$ that is a $c_{1}/2$-corkscrew ball for $\frac{1}{2}\bB\cap \Omega_{j}$ for all $j$ large. 


Let $N$ be the length of the chain. Then this Harnack chain is contained in $\Omega_{j}$ for all $j$ sufficiently large (and by replacing them with some smaller balls, we can replace it with a Harnack chan {\it for} $\Omega_{j}$ of length no more than a multiple of $N$). But this is a contradiction for $j\gg N$.

%
%
%
\end{proof}

\begin{lemma}\label{up}
Let $\Omega\subseteq \bR^{d+1}$ have the CDC and assume \eqref{doubling} holds and $c>0$. Then for any ball $B$ centered on $\d\Omega$ and $B'\subseteq B$ be a $c$-corkscrew ball for $B$ such that $\d\Omega\backslash 4B\neq\emptyset$, there is a Harnack chain of length $N$ (depending on $c$, the CDC and doubling constants) whose first ball is $B'$ and whose last ball is a $c_{1}/2$-corkscrew ball for $2B$. 
\end{lemma}

\begin{proof}
Suppose there were domains $\Omega_{j}$ and balls $B_{j}$ so that $\d\Omega_{j}\backslash 4B_{j}\neq\emptyset$ and a $c$-corkscrew ball $B_{j}'\subseteq B_{j}\cap \Omega_{j}$ so that the length of any Harnack chain from $B_{j}'$ to a $c_{1}/2$-corkscrew ball for $2B_{j}$ is at least $j$. Without loss of generality, $B_{j}=\bB$. Pass to a subsequence just as in the previous lemma, so $B_{j}$ converges to a $c$-corkscrew ball $B(x_{0},c)\subseteq \bB\cap \Omega_{\infty}$. Let $y_{j}\in \d\Omega_{j}\backslash 4\bB$. Since we are assuming \eqref{doubling} holds
\[
\omega_{\Omega_{j}}^{x_{0}}(\Omega\backslash 3\bB)\geq \omega_{\Omega_{j}}^{x_{0}}(B(y_j,|x-y_j|/4))
\stackrel{\eqref{doubling}}{\sim} \omega_{\Omega_{j}}^{x_{0}}(B(x,2|x-y_j|))
\stackrel{\eqref{e:bourgain}}{\sim }1.
\]
Thus, there is $\beta\in (0,1)$ so that 
\[
\omega_{\infty}^{x_{0}}(\cnj{2\bB})\leq \liminf_{j\rightarrow\infty} \omega_{\Omega_{j}}^{x_{0}}(3\bB)<\beta.\]

This means there is a curve in $\Omega_{\infty}^{x_{0}}$ from $x_{0}$ to $\Omega\backslash \cnj{2\bB}$, otherwise $\Omega_{\infty}^{x_{0}}\subseteq 2\bB$, and harmonic measure for bounded domains is a probability measure, but $\supp\omega_{\infty}^{x_{0}}=\d\Omega_{\infty}^{x_{0}}\subseteq \cnj{2\bB}$, so that 
\[
1=\omega_{\infty}^{x_{0}}(\bR^{d})=\omega_{\infty}^{x_{0}}(\cnj{2\bB})<\beta<1,
\]
a contradiction. 

If $\d\Omega_{\infty}\subseteq \frac{3}{2}\bB$ (and recall $c_{1}<1/8$), then this curve connects $x_{0}$ to a ball $B'$ of radius $c_{1}$ whose double is contained in  $2\bB\backslash \frac{3}{2}\bB$ (so it is a $\frac{c_{1}}{2}$-corkscrew ball for $2\bB$). Otherwise, if $\d\Omega_{\infty}\not\subseteq \frac{3}{2}\bB$, then as $\Omega_{\infty}^{c_{0}}$ has the $c_{1}$-interior corkscrew property, there is a $c_{1}$-corkscrew ball $B'\subseteq 2\bB\cap \Omega_{\infty}^{x_{0}}$, and since $\Omega_{\infty}^{x_{0}}$ is connected, the curve can be extended to connect $x_{0}$ to this ball. In either case, there is a curve $\gamma\subseteq \Omega_{\infty}^{x_{0}}$ from $x_{0}$ to the center of a $c_{1}$-corkscrew ball $B'\subseteq 2\bB$ for $\Omega_{\infty}^{x_{0}}$. For $j$ large enough, $\frac{1}{2} B'$ is a $c_{1}/2$-corkscrew ball for $\bB$ in $\Omega_{j}$. Also for $j$ large enough, this curve is also contained in $\Omega_{j}$. We can cover $\gamma$ with  boundedly many balls (depending on $\gamma$ but independent of $j$) to form a Harnack chain between $x_{0}$ and $\frac{1}{2}B'$ in $\Omega_{j}$, but this is a contradiction for $j$ large enough. 

\end{proof}

%

We now finish the proof of  Theorem \ref{thmi}. Assume $\omega_{\Omega}$ is doubling. Let $x\in \Omega$, $\xi\in \d\Omega$, and $0<r<\diam \d\Omega$. Let $\zeta\in \d\Omega$ be closest to $x$. We will show that there is a function $N$ as in Lemma \ref{SU}. 

We will also abuse notation below and write $\log$ for $\log_{+}=\max\{0,\log \}$. 


There are two cases to consider:\\

\noindent {\bf Case 1.} Suppose $\delta_{\Omega}(x)=|x-\zeta|\geq 2^{-4} |x-\xi|$ and $B=B(\xi,2|x-\xi|)$. Since $\omega$ is doubling, $x$ is a $\frac{1}{2^{6}}$-corkscrew point for the ball $B(\xi,2|x-\xi|)$, and if $r<4|x-\xi|$, we can iterate Lemma \ref{down} to find a Harnack chain from $x$ to a $c_{1}$-corkscrew point $y\in B(\xi,r)\cap \Omega$ of length at most a constant times $\log \frac{|x-\xi|}{r}+1$. 

If $r\geq 4|x-\xi|$, then $|x-\xi|\leq r/4<\diam \d\Omega/4$ and since $x$ is a corkscrew point for $B$, we can iterate using Lemma \ref{up} instead to find a Harnack chain to a corkscrew point $y\in B(\xi,r)$ with $|y-x|\gec r$ of length at most a constant times 
\[
\log\frac{r}{|x-\xi|}+1
\leq \log\frac{r}{\delta_{\Omega}(x)}+1
\lec \log\frac{|x-y|}{\delta_{\Omega}(x)}+1.\]

\noindent {\bf Case 2.} Suppose $|x-\zeta|< 2^{-4} |x-\xi|$. Let $B=B(\zeta,2|x-\zeta|)$, so $x$ is a $\frac{1}{2}$-corkscrew point for this ball. Let $k$ be the largest integer for which $\xi\not\in 2^{k+2}B$.  Since 
\[
|\xi-\zeta|
\geq |\xi-x|-|x-\zeta|
>(2^{4}-1)|x-\zeta|
\geq 2^{3}|x-\zeta|=2^{2}r_{B},\]
we know $k\geq 0$. Since $x$ is a $\frac{1}{4}$-corkscrew point for $B$, by iterating Lemma \ref{up}, for $0\leq j\leq k$, we can Harnack chains from a $c_{1}$-corkscrew point in $2^{j}B$ (that is $x$ if $j=0$) to a $c_{1}$-corkscrew point in $2^{j+1}B$ of lengths at most some constant $N$ (depending on the CDC and doubling constants). If we combine these balls, we get a Harnack chain from $x$ to a $c_{1}$-corkscrew point $x'\in 2^{k+1}B$ of total length at most
\[
N\cdot (k+1)\lec \log\frac{|x-\xi|}{|x-\zeta|}+1
= \log\frac{|x-\xi|}{\delta_{\Omega}(x)}+1.
\] 
In particular, since $x'$ is a $c_{1}$-corkscrew point in $2^{k+1}B$ and $\xi\in 2^{k+3}B$, we know $|x'-\xi|\leq 2^{k+6}r_{B}$, and so $x'$ is a corkscrew point in $B(\xi,2^{k+7}r_{B})$. Note $2^{k+7}r_{B}\sim |x-\xi|$. Indeed,

\[
2^{k+3}r_{B}>|\xi-\zeta|\geq |\xi-x|-|x-\zeta|> (1-2^{-4})|\xi-x|>(1-2^{-4})r\]
and 
\[
2^{k+2}r_{B}\leq |\xi-\zeta|\leq  |\xi-x|+|x-\zeta|< (1+2^{-4})|\xi-x|.\]


Hence, $r<2^{k+4}r_{B}$, so we can apply Lemma \ref{down} again, and using the fact that $2^{k}r_{B}\sim |x-\xi|$, we can make a Harnack chain from $x'$ to a corkscrew point $y\in B(\xi,r)$ with $|y-x|\gec r$ of length at most a constant times
\[
\log \frac{2^{k+7}r_{B}}{r} +1
\lec \log \frac{|x-\xi|}{r}+1
\]
Combing our two chains together gives us a Harnack chain from $x$ to $\xi$ of total length at most a constant times $\log \frac{|x-\xi|}{\min\{\delta_{\Omega}(x),r\}}+1$. 
\[
|x-\xi|\leq |x-y|+|y-\xi|
\leq |x-y|+r\lec |x-y|.
\]
This means the chain has length at most $\log \frac{|x-y|}{\min\{\delta_{\Omega}(x),r\}}+1$

Taking the minimum of all the possible estimates we have for the possible length of a Harnack chain gives us our desired function $N$ and semi-uniformity now follows from Lemma \ref{SU}.

\begin{remark}
Note that as a corollary of the proof, we have that, for a semi-uniform domain, the function $N$ given in Theorem \ref{SU} is $N(x)=C\min\{1,\log x+1\}$, that is, the shortest Harnach chain from $x$ to a corkscrew point $y\in B(\xi,r)$ is at most a constant times
\[
\log \frac{|x-y|}{\min\{\delta_{\Omega}(x),r\}}+1. 
\]
\end{remark}

\section{Proof of Theorem \ref{thmiii}: Part I}

\begin{definition}
For a domain $\Omega\subseteq \bR^{d+1}$, we say that points $y_{1},...,y_{n}$ are reference points for a ball $B$ centered on $\d\Omega$ if
\begin{equation}
\label{e:pseudo}
\min_{i=1,...,n}k_{\Omega}(x,y_{i})\lec \log\frac{r_{B}}{\delta_{\Omega}(x)}+1 \;\; \mbox{ for all }x\in B. 
\end{equation}
where $k_{\Omega}(x,y)$ denotes the quasihyperbolic distance between $x$ and $y$. As observed in \cite[p. 434]{AH08}, if $N_{\Omega}(x,y)$ denotes the length of the shortest Harnack chain between $x$ and $y$, then
\[
N_{\Omega}(x,y)\sim k_{\Omega}(x,y)+1.
\]
 
\end{definition}

\begin{remark}
\label{r:ref}
We first make some observations about reference points. 
\begin{enumerate}
\item If $y_{i}$ is a reference point for $B$, then  $\delta_{\Omega}(y_{i})\gec r_{B}$. 
\item By semi-uniformity, if $y_{1},...,y_{n}$ are reference points for $B$ that aren't necessarily in $B$, then using Harnack chains we can find new reference points $z_{1},...,z_{n}$ that are corkscrew points in $B$ (with different corkscrew and reference point constants). This is because semi-uniformity implies we may find Harnack chains from the $y_i$ to corkscrew points $z_i\in B$, so for $x\in B$,
\[
k_{\Omega}(x,z_i)
\leq k_{\Omega}(x,y_i)+k_\Omega(y_i,z_i)
\lec k_{\Omega}(x,y_i)+1.
\]

\item Similarly, if $y_{1},...,y_{n}$ are in $B$ already, $M\geq 1$, and $2Mr_{B} < \diam \d\Omega$, we can also find reference points $z_{1},...,z_{n}$ for $B$ outside $MB$, say, though with constants depending also on $M$. Indeed, if $z\in \d\Omega\backslash 2MB$, by Theorem \ref{SU}, we can find bounded Harnack chains from each $y_{i}$ to a corkscrew ball $z_{i}$ for $B(z,Mr_{B})\subseteq MB^{c}$, so now \eqref{e:pseudo} holds with the $z_{i}$ in place of the $y_{i}$ with constant depending on $M$. 

\item If $B$ is a ball centered on $\d\Omega$ with $r<2\diam \Omega$, we can always find $n$ reference points with $n$ at most a constant depending on the semi-uniformity. Aikawa, Hirata, and Lundh showed this held for any John domain \cite[Proposition 2.1]{AHL06}. For a general semi-uniform domain (which won't be John if it is unbounded), we prove this as follows. First, we can assume that $r_{B}<\diam \Omega/8$. Semi-uniformity and Theorem \ref{SU} imply that for any $x_{0}\in B\cap \Omega$, $\zeta\in \d\Omega$ the closest point to $x_0$ and $B'=B(\zeta,2\delta_{\Omega}(x_{0}))$, and  $k$ the maximal integer so that $2^{k}B'\subseteq 2B$, there are corkscrew points $x_{i}\in 2^{i}B'$ and Harnack chains from $x_{i}$ to $x_{i+1}$ of length at most some number $N$. Note that 
\[
k\sim \floor{\log \frac{r_{B}}{r_{B'}}}+1\lec \log\frac{r_{B}}{\delta_{\Omega}(x_0)}+1\]
and $\delta_{\Omega}(x_{k})\gec  2^{k}r_{B'}\gec r_{B}$. Let $Q_{x_0}$ be the Whitney cube in $\Omega$ to which $x_{k}$ belongs. Then $\ell(Q_{x_0})\sim r_{B}$, $Q_{x_0}\cap 2B\neq\emptyset$, and we have shown that there is a Harnack chain from the center of $Q_{x_0}$ to $x_0$ of length at most a constant times $\log\frac{r_{B}}{\delta_{\Omega}(x_0)}+1$. Since the number of Whitney cubes $Q$ satisfying $\ell(Q)\sim r_{B}$ and $Q\cap B\neq\emptyset$ is uniformly bounded, we can take their centers as our reference points for $B$. 

If $2\diam \d\Omega>r_{B}\geq \diam \d\Omega/8$, we can cover $\d\Omega$ with a bounded number of balls of radius $\diam \d\Omega/16$ and then the union of their respective reference points are a set of reference points for $B$. 

Note that this is not always possible for $r\gg \diam \d\Omega$, and the example is the same as Remark \ref{r:referee}.
\end{enumerate}
\end{remark}

The objective of this section is to prove the following:

\begin{lemma}
\label{l:mainhmlemma}
Let $\Omega\subseteq \R^{d+1}$ be a CDC semi-uniform domain, $B'\subseteq B$ two balls centered on $\d\Omega$ with $r_{B}<\diam \d\Omega$ and $y_{1},...,y_{n}\in \Omega$ a set of reference points for $2B$. There is $M>0$ depending on the CDC and semi-uniformity constants so that 
\[
\frac{\omega_{\Omega}^{x}(B')}{\omega_{\Omega}^{x}(B)}\lec \sum_{i=1}^{n} \omega_{\Omega}^{y_{i}}(B') \;\; \mbox{ for all }x\in \Omega\backslash MB .\]
In particular, if $E\subseteq B$ is a Borel set, then 
\[
\frac{\omega_{\Omega}^{x}(E)}{\omega_{\Omega}^{x}(B)}\lec \sum_{i=1}^{n} \omega_{\Omega}^{y_{i}}(E) \;\; \mbox{ for all }x\in \Omega\backslash MB.\]
\end{lemma}

We recall the following lemma from \cite[Lemma 3.6]{AH08}. The statement there is slightly different, but the proof is exactly the same.

\begin{lemma}
\label{l:AHlemma}
Let $\Omega\subseteq \R^{d+1}$ be a CDC domain, $B$ a ball centered on $\d\Omega$ with $r_{B}<\diam \d\Omega$, and $y_{1},...,y_{n}\in \Omega$ be reference points for $2B$. Then
\begin{equation}
\omega_{\Omega}^{x}(B)\lec r_{B}^{d-1} \sum_{i=1}^{n}G_{\Omega}(x,y_{i}) \;\; \mbox{ for }x\in \Omega\backslash 2B.
\end{equation}
The implied constant depends on the CDC constant and reference point constants. 
\end{lemma}

\begin{remark}
Note that by Remark \ref{r:ref}, the reference points can also be taken to be corkscrew points in $B$.
\end{remark}

\begin{lemma}
\label{l:carleson}
Let $\Omega\subseteq \R^{d+1}$ be a CDC domain and $B$ a ball centered on $\d\Omega$ with $r_{B}<\diam \d\Omega$. Let $y_{1},...,y_{n}$ be reference points for $2B$, and $M$ large enough (depending on the reference point constants) so that each $x\in 2B\cap \Omega$ can be connected to one of the $y_{i}$ by a Harnack chain of length $N_{\Omega}(x,y_i)$ in $MB$. Let $u$ be a non-negative harmonic function on $\Omega\cap MB$ vanishing continuously on $2B\cap \d\Omega$.  then
\[
\sup_{B\cap \Omega} u\lec \sum_{j=1}^{n}  u(y_{i}). 
\]
\end{lemma}

\begin{proof}
The proof of this is almost exactly like that of \cite[Lemma 4.4]{JK82}, we just point out the required modifications in its proof. First, \cite[Lemma 4.1]{JK82} still holds in CDC domains, as it is just Lemma \ref{l:holder}. In particular, there is $M_{1}$ so that for any $\xi\in \d\Omega$ and $s<r_{B}$, 
\[
\sup\{u(x): x\in B(\xi,M_{1}^{-1}s)\cap \Omega \}<\frac{1}{2} \sup\{u(x): x\in B(\xi,s)\cap \Omega \}.
\]

Next, if we assume $\max u(y_{i})=1$, then using Harnack chains and the fact that the $y_{1},...,y_{n}$ are reference points for $2B$, one can show that there is $M_{2}>1$ depending on $M_{1}$ and the reference point constants so that if $u(y)>M_{2}^{h}$ for an integer $h$ and $y\in 2B\cap \Omega$, then $\delta_{\Omega}(y)<M_{1}^{-h}$. The proof now follows that of \cite[Lemma 4.4]{JK82} word by word.
\end{proof}

A domain $\Omega$ satisfies the {\bf boundary Harnack principle (BHP)} if there is $M\geq 1$ so that, if $u,v$ are non-negative harmonic functions vanishing continuously on $MB\cap \d\Omega$ and $x_{0}$ is a corkscrew point in $B$, then
\[
\frac{u(x)}{v(x)}\sim \frac{u(x_{0})}{v(x_{0})} \;\; \mbox{for all} \;\; x\in B\cap \Omega.
\]
This is shown for NTA domains in \cite[Lemma 4.10]{JK82} and was a key ingredient in Jerison and Kenig's proof of \eqref{e:jkw}, see \cite[Lemma 4.11]{JK82}. However, Aikawa has shown that, if $\Omega$ is a CDC John domain, the BHP is equivalent to $\Omega$ being a uniform domain \cite{Aik06}, so we can't expect such an estimate to hold in our setting. The following lemma serves as a weak substitute for the BHP in semi-uniform domains, and its proof is based on that of \cite[Lemma 4.10]{JK82}. 

\begin{lemma}
Let $\Omega\subseteq \R^{d+1}$ be a semi-uniform CDC domain. Then there are constants $2M_{0}<M_{1}$ depending on the semi-uniformity constants so that the following holds. Let $B$ a ball centered on $\d\Omega$ with $r_{B}<\diam \d\Omega$, and $y_{1},...,y_{n}\in B$ be reference points for $B$.  Let $u$ be a non-negative harmonic function on $M_{1}B\cap \Omega$ that vanishes continuously on $2M_{0}B\cap \d\Omega$. Then
\begin{equation}
\label{e:bhp}
u(x)\lec \ps{\sum_{i=1}^{n} u(y_{i})} \sum_{i=1}^{n} r_{B}^{d-1}G_{\Omega}(x,y_{i}) \;\; \mbox{ for all }x\in B. 
\end{equation}
\end{lemma}

\begin{proof}

Let $\cW$ denote the Whitney cubes in $\Omega$. Pick $M_{0}$ large enough (depending on the semi-uniformity constants) so that each $x\in 2B\cap \Omega$ can be connected to one of the $y_{i}$ by a Harnack chain of length $N_{\Omega}(x,y_i)$ so that $2Q\subseteq M_{0}B$ for each Whitney cube $Q$ for $\Omega$ that intersects the Harnack chain. Note that by semi-uniformity, we also know that the $y_{1},...,y_{n}$ are also reference points for $2M_{0}B$ with different reference point constants: each point $x\in 2M_{0} B\cap \Omega$ can be connected by a Harnack chain of length $\lec \log\frac{2M_{0}r_{B}}{\delta_{\Omega}(x)}+1$ to a corkscrew point in $B$, and this can be extended by a bounded number of balls to one of the $y_{i}$ of total length $N_{\Omega}'(x,y_{i})\sim N_{\Omega}(x,y_{i})$, say. Now pick $M_{1}$ large enough (depending on the new reference point constants) so that each $x\in 2M_{0}B\cap \Omega$ can be connected to a $y_{i}$ by a Harnack chain of length $N_{\Omega}'(x,y_{i})$ contained in $M_{1}B$.

Without loss of generality, $\sum_{i=1}^{n} u(y_{i})=1$. For each $Q\in \cW$ intersecting $2B$, there is $\{R_{i}^{Q}\}_{i=1}^{N_{Q}}$ a chain of cubes, the first containing one of the $y_{j}$, the last equalling $Q$, where 
\[
N_{Q}\lec 1+\log\frac{r_{B}}{\ell(Q)}.\] 
Let $\lambda>1$ be small and
\[
\Omega' := \bigcup_{Q\in \cW \atop Q\cap 2B\neq\emptyset} \bigcup_{i=1}^{N_{Q}} \lambda R_{i}^{Q} \supseteq 2B.
\]
Note that by construction and our choice of $M_{0}$ that
\[
2B\cap \Omega \subseteq \Omega' \subseteq M_{0}B\cap \Omega.
\]

By Lemma \ref{l:carleson} and our choice of $M_{1}$, 
\[
\sup_{M_{0}B\cap \Omega} u\lec \sum u(y_{i})\leq 1.\]
In particular, by the maximum principle we have
\begin{equation}
\label{e:u<w}
u(x) \leq \omega_{\Omega'}^{x}(\d\Omega'\cap \Omega) \;\;\mbox{ for all }x\in \Omega'.
\end{equation}
Let $\{B_{i}\}_{i=1}^{m}$ be a finite collection of balls (with $m$ depending only on $d$) of radius $\eta r_{B}$ (where $\in (0,\eta)$ will be chosen later) centered along $\d\Omega'$ that cover 
\[
L=\{x\in \d\Omega': \delta(x)<\eta^{2}r_{B}\}.
\]

See Figure \ref{f:outball}.

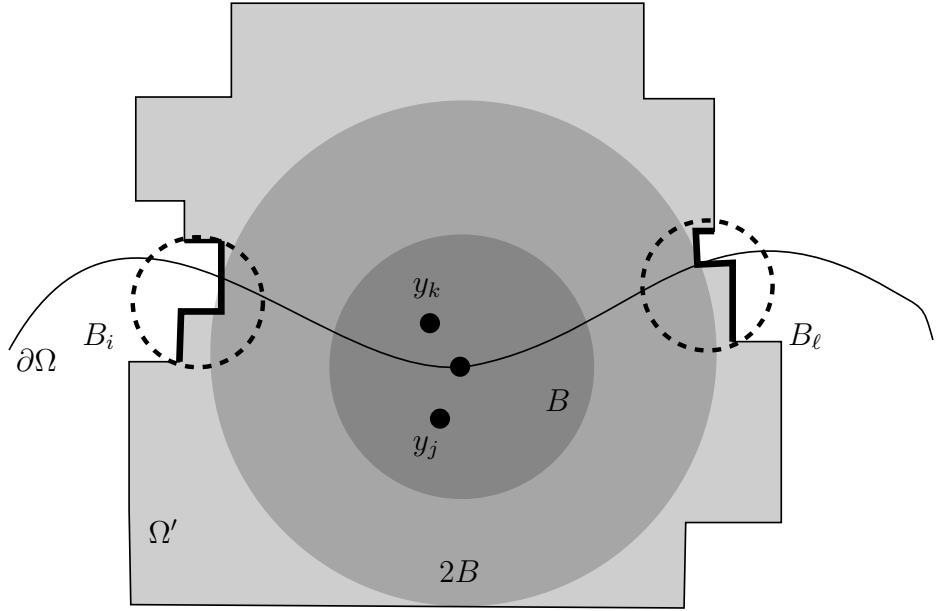
\begin{figure}[!ht]
\includegraphics[width=350pt]{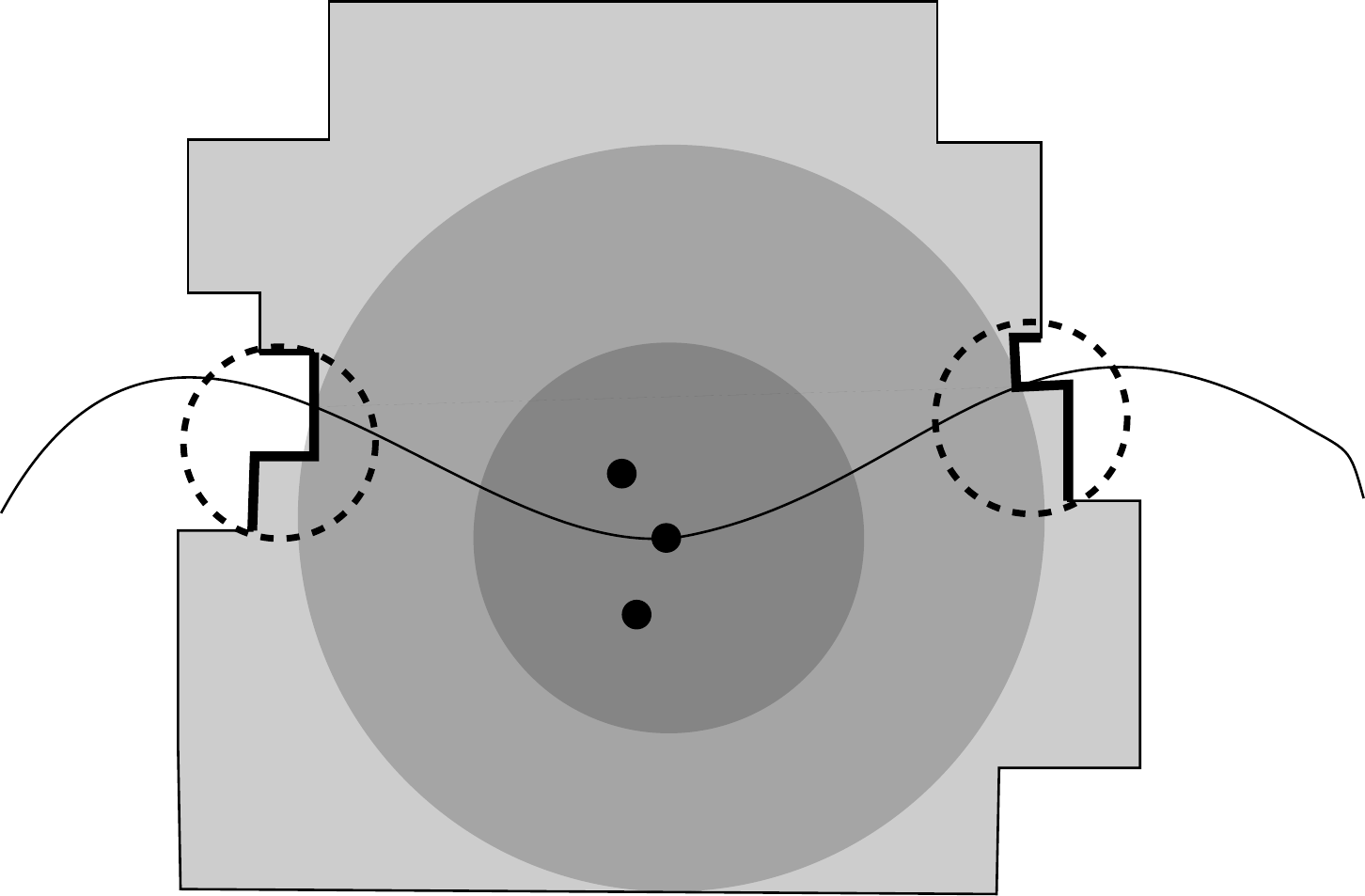}
\begin{picture}(0,0)(350,0)
\put(25,100){$B_{i}$}
\put(50,25){$\Omega'$}
\put(160,10){$2B$}
\put(200,75){$B$}
\put(292,100){$B_{\ell}$}
\put(150,60){$y_{j}$}
\put(150,120){$y_{k}$}
\put(0,90){$\d\Omega$}
\end{picture}
\caption{The dark lines denote the set $L\subseteq \d\Omega'$.}
\label{f:outball}
\end{figure}

{\bf Claim:} for each $i$, the points $y_{1},...,y_{n}$ are reference points for $2B_{i}$ with respect to $\Omega'$. 

Let $x\in B_{i}\cap \Omega'$, so $x\in \lambda Q$ for some $Q\in \cW$. 
Then $Q=R_{i}^{Q''}$ for $i$ and some cube $Q''\in \cW$ such that $Q''\cap B\neq\emptyset$ and $R_{1}^{Q''}$ contains $y_{j}$ for some $j\in \{1,2,...,n\}$. Moreover, $i\lec 1+\log \frac{r_{B}}{\ell(Q)}$. Thus, we can find a Harnack chain in $\bigcup_{k=1}^{i} \lambda R_{k}^{Q''} \subseteq \Omega'$ from $y_{j}$ to $x_{Q}$ (the center of $Q$) of length at most a constant times $1+\log\frac{r_{B}}{\ell(Q)}$. If $\delta_{\Omega'}(x)\geq \ve  \ell(Q)$, then  we can add on a finite number (depending on $\ve$) of balls to complete this chain into a Harnack chain from $y_{j}$ to $x$. If $\delta_{\Omega'}(x)<\ve \ell(Q)$, then for $\ve>0$ small enough (depending on $\lambda$), $\delta_{\lambda Q}(x)\sim \delta_{\Omega'}(x)$, and since $\lambda Q$ is chord-arc, we can connect $x$ to $x_{Q}$ by a Harnack chain of length at most a constant times
\[
\log \frac{\ell(Q)}{\delta_{\lambda Q}(x)}+1\sim  \log \frac{\ell(Q)}{\delta_{\Omega'}(x)}+1\]
and then we can connect $y$ to $x_{Q}$ by a Harnack chain of length at most a constant times $ \log \frac{r_{B}}{\delta_{\Omega'}(x)}+1$.

In either case, by summing up the estimates for the lengths, we obtain a Harnack chain of length at most $\log \frac{r_{B}}{\delta_{\Omega'}(x)}+1$ from $x$ to some $y_{i}$. This completes the claim. 

For $\eta$ small enough, we can ensure that, for each $j$, $B\cap 2B_{j}=\emptyset$, and so Lemma \ref{l:AHlemma}, the maximum principle implies that for all $x\in B$,
\begin{equation}
\label{e:omegaL1}
\omega_{\Omega'}^{x}(L)
\leq \sum_{j=1}^{m}\omega_{\Omega'}^{x}(B_{j})
\lec \sum_{j=1}^{m}\sum_{i=1}^{n} r_{B_{j}}^{d-1} G_{\Omega'}(x,y_{i})
\lec \sum_{i=1}^{n} r_{B}^{d-1} G_{\Omega}(x,y_{i}).
\end{equation}

If $x\in \d\Omega'\backslash L$, then by construction, there is a short Harnack chain from $x$ to one of the $y_{i}$, and thus to a point $z_{i}\in \d B(y_{i},\ve cr_{B})$, where $c$ is the corkscrew constant and $\ve>0$ is small enough so that $B(y_{i},2\ve cr_{B})$ are disjoint (we can replace them with other reference points so that this happens by Remark \ref{r:ref}). Since 
\[
 1 
\lec r_{B}^{d-1}G_{\Omega}(z_{i},y_{i})
\lec r_{B}^{d-1}G_{\Omega}(x,y_{i}),\]
we have $\sum_{j} r_{B}^{d-1}G_{\Omega}(x,y_{j})\gec 1$ for $x\in \d\Omega'\backslash L\cup \bigcup \d B(y_{j},\ve cr_{B})$, and so by the maximum principle on $\Omega'\backslash \bigcup \d B(y_{j},\ve cr_{B})$ that
\[
\omega_{\Omega'}^{x}(\d\Omega'\backslash L)\lec \sum_{i=1}^{n}  r_{B}^{d-1} G_{\Omega}(x,y_{i}) \;\; \mbox{ for all }x\in \Omega'.
\]
Combining the above estimates and using the bounded overlap of the $B_{j}$, we obtain that for $x\in B\cap \Omega$,
\[
u(x)
\stackrel{\eqref{e:u<w}}{\leq} 
\omega_{\Omega'}^{x}(\d\Omega\cap \Omega)
\leq \omega_{\Omega'}^{x}(L) + \omega_{\Omega'}^{x}(\d\Omega'\backslash L)
\lec   \sum_{i=1}^{n}r_{B}^{d-1} G_{\Omega}(x,y_{i})
\]
and this gives \eqref{e:bhp}.
\end{proof}

\begin{proof}[Proof of Lemma \ref{l:mainhmlemma}]
Let $y_{1},...,y_{n}\in \frac{1}{4}B$ be reference points for $B$ and $z_{1},...,z_{n}\in B'$ be reference points for $B'$. Then for $x\in \Omega\backslash MB$ and $M$ large enough, Lemma \ref{l:AHlemma} implies
\[
\omega_{\Omega}^{x}(B')
\lec r_{B'}^{d-1}  \sum_{i=1}^{n} \underbrace{G_{\Omega}(x,z_{i})}_{=: u(z_{i})}.
\]
Note that by \eqref{e:bhp},
\[
u(z_{i})
\lec \ps{\sum_{j=1}^{n} u(y_{j})}\ps{\sum_{j=1}^{n} r_{B}^{d-1} G_{\Omega}(z_{i},y_{j})}
\]
By \eqref{w>G} and the doubling property,
\[
u(y_{i}) = G_{\Omega}(x,y_{i}) 
\lec r_{B}^{1-d} \omega_{\Omega}^{x}(B).
\]
And similarly, 
\[
G_{\Omega}(z_{i},y_{j})
\lec r_{B'}^{1-d}\omega_{\Omega}^{y_{j}}(B').\]

So the above estimates combined give
\[
\omega_{\Omega}^{x}(B')
\lec \omega_{\Omega}^{x}(B)\sum_{i=1}^{n} \omega_{\Omega}^{y_{i}}(B').\]

\end{proof}

\begin{remark}
It's natural to ask whether we can get away with just one $y_{j}$ in the above estimate instead of having to sum over all reference points, or in other words, whether $\omega_{\Omega}^{x_{1}}(B)\sim \omega_{\Omega}^{x_{2}}(B)$ for any two corkscrew points $x_{1},x_{2}\in B$, but this is not the case. If we consider the von Koch snowflake, or any  NTA domain $\Omega_{1}$ whose exterior domain $\Omega_{2}=(\Omega_{1}^{c})^{\circ}$ is also NTA and whose common boundary is purely unrectifiable boundary, then $\omega_{\Omega_{1}}^{x_{1}}$ and $\omega_{\Omega_{2}}^{x_{2}}$ must be mutually singular. If we remove a ball $B_{0}$ from the boundary of this domain, we now have a connected domain $\Omega=\Omega_{1}\cup \Omega_{2}\cup B_{0}$, and if $x_{1}$ and $x_{2}$ are two corkscrew points for some fixed ball $B$ centered on the boundary away from $B_{0}$ (for $\Omega_{1}$ and $\Omega_{2}$ respectively), then as the ball shrinks, $\omega_{\Omega}^{x_{i}}\warrow \omega_{\Omega_{i}}^{x_{i}}$. In particular, we can find a ball $B'\subseteq B$ so that $\omega_{\Omega_{1}}^{x_{1}}(B') \ll \omega_{\Omega_{2}}^{x_{2}}(B')$, say, and then this will imply, for $B_{0}$ small enough, $\omega_{\Omega}^{x_{1}}(B') \ll \omega_{\Omega}^{x_{2}}(B')$.
\end{remark}

\section{Proof of Theorem \ref{thmiii}: Part II}

The objective of this section is to prove the counterpart to Lemma \ref{l:mainhmlemma}.

\begin{lemma}
\label{l:mainhmlemma2}
Let $\Omega\subseteq \R^{d+1}$ be a semi-uniform CDC domain,  $B$ a centered on $\d\Omega$ with $r_{B}<\diam \d\Omega/4$, $E\subseteq \d\Omega\cap B$ Borel, and $x_{1},...,x_{n}\in \Omega$ a set of reference points for $B$. There is $M$ depending on the CDC and doubling constants and integer $i$ so that 
\[
\frac{\omega_{\Omega}^{x}(E)}{\omega_{\Omega}^{x}(B)}\gec  \min_{i=1,,...,n}\omega_{\Omega}^{x_{i}}(E) \mbox{ for all }x\in \Omega\backslash MB.
\]

\end{lemma}

For this, we will need different estimates on Green's function. 

\begin{lemma}
Let $\d\Omega\subseteq \R^{d+1}$ be a semi-uniform domain and $B$ a ball centered on $\d\Omega$. Let $x_{1},...,x_{n}\in \Omega\backslash 2B$ be reference points for $B$, and for each $x\in \Omega$, let $\{Q_{i}(x)\}_{i=1}^{N(x)}$ be a Harnack chain of cubes from one of the $x_{1},...,x_{n}$ to $x$ so that $N(x)\lec  \log \frac{r_{B}}{\delta_{\Omega}(x)}+1$ (which exists by semi-uniformity). Let
\[
\Omega'=\bigcup_{x\in \Omega\backslash B}\bigcup_{i=1}^{N(x)}\lambda Q_{i}(x).
\]
Then $x_{1},...,x_{n}$ are reference points for $\Omega'\cap 2B$. 
\end{lemma}

\begin{proof}
Let $x\in \Omega'\cap 2B$. Then $x\in \lambda Q_{i}(y)\cap 2B$ for some $y\in \Omega\backslash B$, where $ Q_{1}$ contains some reference point $x_{j}$. Let $Q=Q_{i}(y)$. Then there is a Harnack chain from $x_{Q}$ to $x_i$ of length $
\lec \log \frac{r_{B}}{\ell(Q)}+1$ (just by following the chain of $Q_{k}$ back up to $x_{j}$).
%
If $\delta_{\Omega'}(x)\geq  \ve\ell(Q)$ for some $\ve>0$, then it is easy to find a Harnack chain of bounded length from $x$ to $x_{Q}$ (of length depending on $\ve$). If $\delta_{\Omega'}(x)< \ve \ell(Q)$, then for $\ve>0$ small enough, $\delta_{\Omega'}(x)\sim \delta_{\lambda Q}(x)$ and there is a Harnack chain in $\lambda Q$ from $x$ to $x_{Q}$ of length at most a constant times
\[
\log\frac{\ell(Q)}{\delta_{\lambda Q}(x)}+1
\lec \log\frac{\ell(Q)}{\delta_{\Omega'}(x)}+1.\]
Connecting these chains give a Harnack chain from $x$ to $x_{i}$ of length at most $C\log\frac{r_B}{\delta_{\Omega'}(x)}+1$, and this finishes the proof. 
\end{proof}

\begin{proof}[Proof of Lemma \ref{l:mainhmlemma2}]
Let $x\in \Omega\backslash  4B$. By Remark \ref{r:ref}, since $4r_{B}<\diam \d\Omega$, we can assume that our reference points are outside $2B$. Let $y_{i}\in \d \Omega'$ be a corkscrew point in $\Omega$ accessible from $x_{i}$, so that $\delta_{\Omega}(y_{i})\sim r_{B}$ and there is a Harnack chain of bounded length from $y_{i}$ to $x_{i}$ (we can find these since semi-uniformity implies the existence of a Harnack chain from $y_{i}$ to the center of $B$, and so one of these balls must cross $\d\Omega'$). See Figure \ref{f:inball}. 

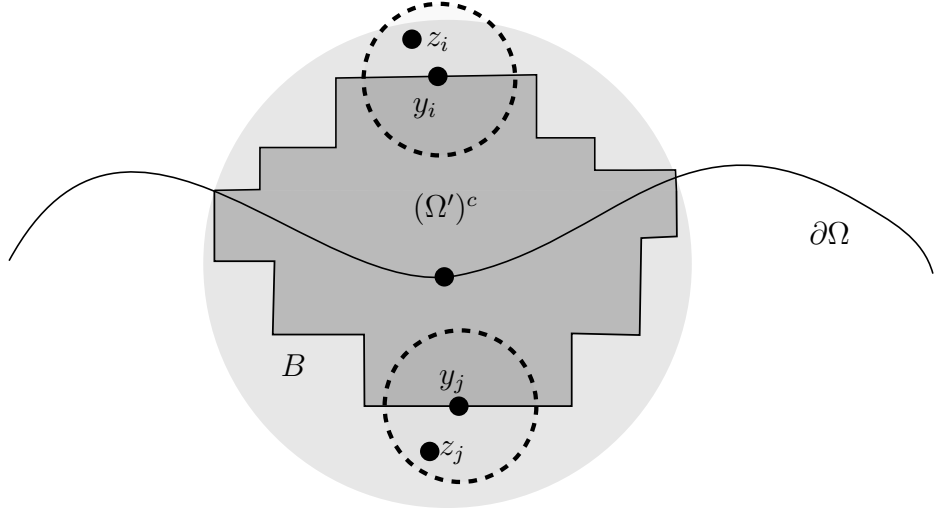
\begin{figure}[!ht]
\includegraphics[width=350pt]{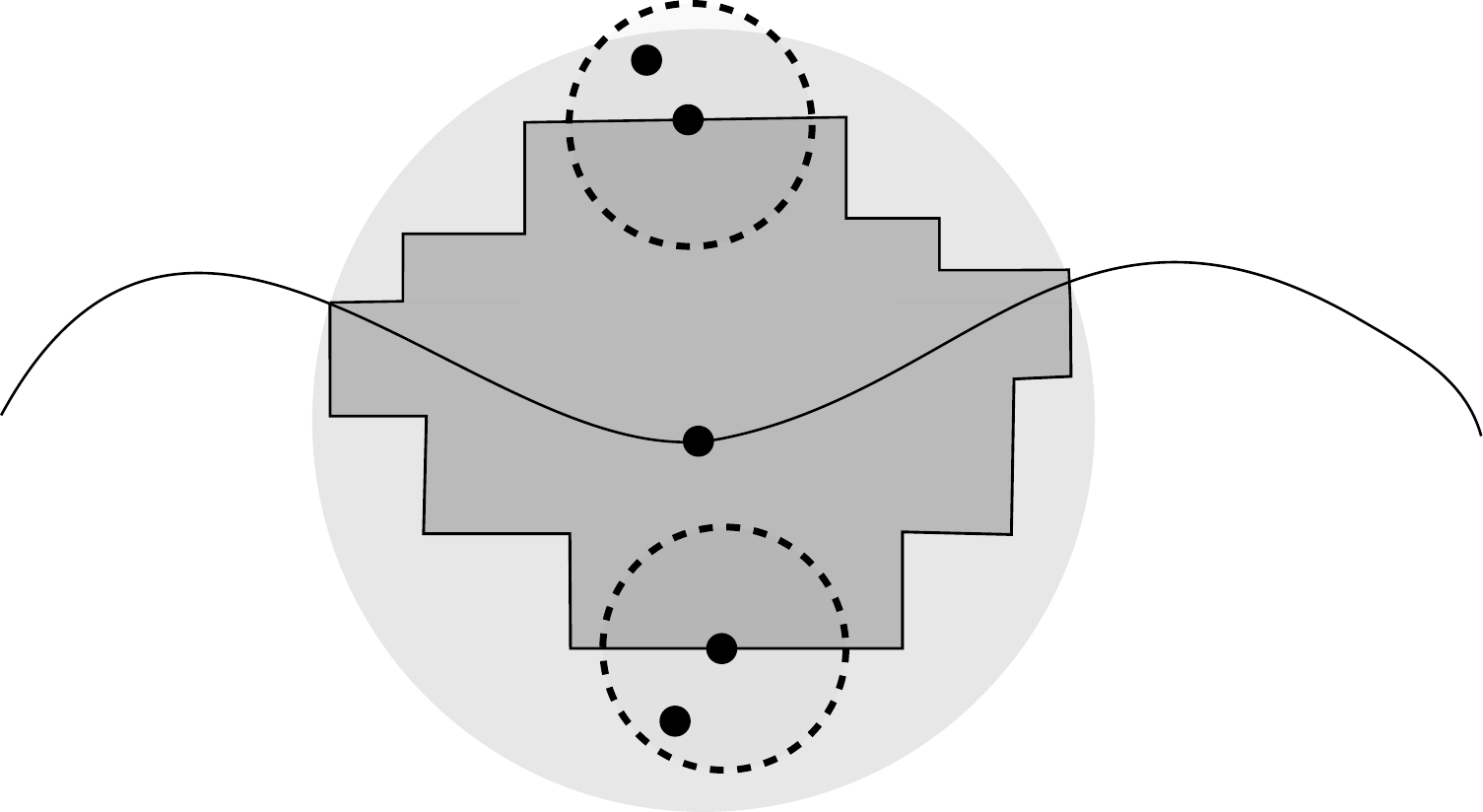}
\begin{picture}(0,0)(350,0)
\put(150,110){$(\Omega')^{c}$}
\put(150,150){$y_{i}$}
\put(100,50){$B$}
\put(300,100){$\d\Omega$}
\put(160,20){$z_{j}$}
\put(155,175){$z_{i}$}
\put(160,47){$y_{j}$}
\end{picture}
\caption{Observe that $\d\Omega\cap B\subseteq (\Omega')^{c}$, so by the maximum principle, $\omega_{\Omega'}^{x}(B)\geq \omega_{\Omega}^{x}(B)$ for $x\in \Omega'$.}
\label{f:inball}
\end{figure}

In particular, we may find a Harnack chain from $x_{i}$ to a corkscrew point $z_{i}\in B(y_{i},\frac{\delta_{\Omega}(y_{i})}{10})\cap \Omega'$. 

By the previous lemma, $x_{1},...,x_{n}$ are all reference points in $\Omega'$ for $2B$. By Lemma \ref{l:AHlemma} (using the fact that $x\not\in 2(2B)$) and using Harnack chains from $x_{i}$ to $z_{i}$ and the fact that $r_{B}\sim \delta_{\Omega}(y_{i})$,
\begin{align*}
\omega_{\Omega_{B}}^{x}({2B})
& \lec r_{B}^{d-1} \sum_{i=1}^{n} G_{\Omega'}(x,x_{i})
\sim  r_{B}^{d-1} \sum_{i=1}^{n} G_{\Omega'}(x,z_{i})\\
& \stackrel{\eqref{w>G}}{\lec} \sum_{i=1}^{n} \omega_{\Omega'}^{x}(B(y_{i},\delta_{\Omega}(y_{i})/2)).
\end{align*}
In particular, there is $j$ so that 
\begin{equation}
\label{e:w>w}
 \omega_{\Omega'}^{x}(B(y_{j},\delta_{\Omega}(y_{j})/2))\gec \omega_{\Omega'}^{x}(2B).
 \end{equation}
 Using Harnack chains, we know $\omega_{\Omega}^{y}(E)\sim \omega_{\Omega}^{y_{i}}(E)$ for all $y\in B(y_{i}, \delta_{\Omega}(y_{i})/2)$. Also, since $\d\Omega \cap B\subseteq (\Omega')^{c}$, we know $\omega_{\Omega}(B)$ vanishes continuously along $\d\Omega'\backslash \cnj{B}$ and is at most $1$ on $\cnj{B}\cap \d\Omega'$. Thus, by the maximum principle,
 \begin{equation}
 \label{e:w<wb}
\omega_{\Omega}^{z}(B)\leq \omega_{\Omega'}^{z}(\cnj{B}) \leq \omega_{\Omega'}^{z}(2{B})  \mbox{ for all }z\in \Omega'.
\end{equation}

Hence, by the strong Markov Property of harmonic measure and the maximum principle
\begin{align*}
\omega_{\Omega}^{x}(E)
&  =\int_{\d \Omega'\cap \Omega} \omega_{\Omega}^{y}(E)d\omega_{\Omega'}^{x}(y)
\geq \int_{B(y_{j},\delta_{\Omega}(y_{j})/2)} \omega_{\Omega}^{y}(E)d\omega_{\Omega'}^{x}(y)\\
& \gec  \int_{B(y_{j},\delta_{\Omega}(y_{j})/2)} \omega_{\Omega}^{y_{j}}(E)d\omega_{\Omega'}^{x}(y)
=\omega_{\Omega}^{y_{j}}(E)\omega_{\Omega}^{x}(B(y_{j},\delta_{\Omega}(y_{j})/2))\\
& \stackrel{\eqref{e:w>w}}{\gec} \omega_{\Omega}^{y_{j}}(E)\omega_{\Omega'}^{x}({2B})
\stackrel{\eqref{e:w<wb}}{\geq}  \omega_{\Omega}^{y_{j}}(E)\omega_{\Omega}^{x}(B)
\sim  \omega_{\Omega}^{x_{j}}(E)\omega_{\Omega}^{x}(B).
\end{align*}

\end{proof}

\begin{proof}[Proof of Theorem \ref{thmiii}]
The first inequality in Theorem \ref{thmiii} follows from Lemma \ref{l:mainhmlemma} by choosing some reference points $y_{1},...,y_{n}$ that are also corkscrew points for $B$ (recall Remark \ref{r:ref}). The second inequality follows from  \ref{l:mainhmlemma2} when $r_{B}<\frac{1}{4}\diam \d\Omega$. If $\diam \d\Omega>r_{B}\geq \diam \d\Omega$, things are a little easier: by Remark \ref{r:ref}, and assuming $0\in \d\Omega$ and $\diam \d\Omega=1$, we can find Reference points $y_{i}$ for $2\bB$ that are contained inside $\bB$. In particular, every $x\in \d 2 B$ can be connected to a reference point $y_{i}$ by a short Harnack chain. Also, for each such $x$, $\omega_{\Omega}^{x}(B)\sim \omega_{\Omega}^{x}(2\bB)\gec 1$ by \eqref{e:bourgain} and the doubling condition. These two observations imply
\[
\omega_{\Omega}^{x}(E)
\gec  \min_{i} \omega_{\Omega}^{y_{i}}(E) 
\gec \omega_{\Omega}^{x}(B) \min_{i}\omega_{\Omega}^{y_{i}}(E) \mbox{ for $x\in \d 2\bB$}
\]
Thus, the second inequality Theorem \ref{thmiii} holds in this case as well by the maximum principle.
\end{proof}

\section{Chord-arc subdomains of semi-uniform domains with UR boundary }

We recall the construction of cubes on a metric space, originally due to David and Christ (\cite{Dav88}, \cite{Christ-T(b)}), but the current formulation is from Hyt\"onen and Martikainen \cite{HM12}. This construction works for any doubling metric space $X$, but we state it for the case $X=\d\Omega$, where $\d\Omega$ is the boundary of some domain. 

{
 \begin{definition}
We say that a set $X$ is {\it $\delta$-separated} or a {$\delta$-net} if for all $x,y\in X$ we have $|x-y|\geq \delta$. 
\end{definition}
}

 \begin{theorem}
Let $X$ be a metric space and let $X_{k}$ be a nested sequence of maximal $\rho^{k}$-nets for $X$ where $\rho<1/1000$ and let $c_{0}=1/500$. For each $n\in\bZ$ there is a collection $\cD_{k}$ of ``cubes,'' which are Borel subsets of $E$ such that the following hold.
\begin{enumerate}
\item For every integer $k$, $X=\bigcup_{Q\in \cD_{k}}Q$.
\item If $Q,Q'\in \cD=\bigcup \cD_{k}$ and $Q\cap Q'\neq\emptyset$, then $Q\subseteq Q'$ or $Q'\subseteq Q$.
\item For $Q\in \cD$, let $k(Q)$ be the unique integer so that $Q\in \cD_{k}$ and set $\ell(Q)=5\rho^{k(Q)}$. Then there is $\zeta_{Q}\in X_{k}$ so that
\begin{equation}\label{e:containment}
B(\zeta_{Q},c_{0}\ell(Q) )\subseteq Q\subseteq B(\zeta_{Q},\ell(Q))=:B_{Q}
\end{equation}
and
\[ X_{k}=\{\zeta_{Q}: Q\in \cD_{k}\}.\]
\end{enumerate}
\item If $\mu$ is a doubling measure on $E$, $\tau\in (0,1)$, $Q\in \cD_{k}$ for some $k$, and 
\[
(1-\tau)Q:=\{x\in Q: \dist(x,E\backslash Q)\geq \tau \ell(Q)\},
\]
then $\mu(Q\backslash (1-\tau)Q)\lec \tau^{\alpha}$ for $\alpha$ and implied constant depending on the doubling constant for $\mu$. 
\label{t:Christ}
\end{theorem}

We next recall a theorem from \cite{HMM14} (which is a bilateral version of the Coronization theorem of David and Semmes \cite{DS}).

\begin{lemma}
\label{l:HMM}
Let $\Omega\subseteq \R^{d+1}$ be a domain with UR boundary, $K\gg 1,0<\eta\ll 1$, and let $\sigma=\cH^{d}|_{\d\Omega}$. Let $\cD$ denote the cubes from Theorem \ref{t:Christ} for $E=\d\Omega$. Then we may partition $\cD=\cG\cup \cB$ such that the following hold:
\begin{enumerate}
\item The cubes in $\cB$ satisfy a Carleson packing condition:
\[
\sum_{Q\subseteq R\atop Q\in \cB} \sigma(Q) \lec \sigma(R) \;\; \mbox{ for all }R\in \cD.\]
\item $\cG=\bigcup_{S\in \cF} S$ where each $S\in \cF$ is a {\it stopping-time region}, meaning 
\begin{enumerate}[(a)]
\item $S$ contains a unique maximal element $Q(S)$ so that $Q\subseteq Q(S)$ for all $Q\in S$.
\item If $Q\in S$ and if $Q\subseteq R\subseteq Q(S)$, then $R\in S$,
\item If $Q\in S$, either all of its children are in $S$ or none of them
are. This last property ensures that every $x\in Q(S)$ is either contained in infinitely many cubes from $S$ or it is contained in a minimal cube that does not properly contains other cubes from $S$. We denote the set of minimal cubes $m(S)$. 
\end{enumerate}
\item The cubes $\{Q(S):S\in \cF\}$ have a Carleson packing condition:
\[
\sum_{Q(S)\subseteq R\atop S\in \cF} \sigma(Q(S)) \lec \sigma(R) \;\; \mbox{ for all }R\in \cD.
\]
\item For each $S\in \cF$, there is a $d$-dimensional $\eta$-Lipschitz graph $\Gamma_{S}$ so that, for all $Q\in S$, 
\begin{equation}
\label{e:closegraph}
\sup_{x\in K^2B_{Q}\cap E}\dist(x,\Gamma_{S})+\sup_{X\in K^2B_{Q}\cap \Gamma_{S}}\dist(x,E)<\eta \ell(Q).
\end{equation}
\end{enumerate}
\end{lemma}

The main objective of this section is to prove the following lemma. 

\begin{lemma}
\label{l:CAD}
Let $\Omega\subseteq \R^{d+1}$ be semi-uniform with UR boundary and $\cD=\cG\cup \cB$ the decomposition as in Lemma \ref{l:HMM}. For each $\delta>0$, $Q_{0}\in \{Q(S):S\in \cF\}$, and $x_{0}$ a corkscrew point for $B_{Q_{0}}$, there is a CAD $\Omega_{0}\subseteq MB_{0}$ (with constants depending on $\delta$ and the UR constants of $\d\Omega$) so that 
\[\sigma(Q_{0}\backslash \d\Omega_{0})<\delta \sigma(Q_0)\]
and
\[
\delta_{\Omega_{0}}(x_{0})\sim \diam \Omega_{0}\sim \ell(Q_{0}). 
\]
\end{lemma}

%


The rest of this section is dedicated to the proof of \ref{l:CAD}. The arguments below take inspiration not just from \cite{BH17} and \cite{HM14}, but also from \cite[Chapter 16]{DS}. We understand that there are many constructions of chord-arc subdomains for uniform domains and NTA domains, and though we are working in semi-uniform domains, the details are similar and in some cases identical to steps in these other cases (see for example \cite{HM14}). However, to spare the reader the task of checking these references to adapt the steps themselves, we present all of them here so there is no ambiguity.\\

First we give a vague sketch of the proof. According to the corona decomposition, we can decompose the cubes in $Q_{0}$ into bad cubes $\cB$ and stopping time regions $S\in \cF$. For each $S\in \cF$, $\d\Omega$ is well approximated by a graph $\Gamma_{S}$ near cubes in $S$. Thus, we can easily construct two Lipschitz domains $\Omega_{S}^{+}$ and $\Omega_{S}^{-}$ above and below the graph whose boundaries are close to $\d\Omega$ near cubes in $S$.  In particular, if $x\in Q_{0}$ is contained in infinitely many $Q\in S$, then $x\in \d\Omega_{S}^{\pm}$. We then connect some of these into one CAD as follows: first, $Q_{0}=Q(S_{0})$ for some stopping-time region $S_{0}$, so we add either of $\Omega_{S_{0}}^{\pm}$ to our CAD (whichever contains the corkscrew point $x_{0}$, and one of them should since $\d\Omega$ is very flat inside $KB_{Q_{0}}$ because of \ref{e:closegraph}). For the stopping-time regions $S$ just below $S_{0}$, semi-uniformity implies we may connect at least one of the $\Omega_{S}^{\pm}$ (say it is $\Omega_{S}^{+}$) up to $\Omega_{S_{0}}$ by a Harnack chain, and we declare $\Omega_{S}$ to be $\Omega_{S}^{+}$ plus this chain. We continue so forth, adding CADs $\Omega_{S}$ for $N$ levels of stopping-time regions $S$ to construct one large CAD $\Omega_{0}$. We pick $N$ large enough so that, by the Carleson packing condition, most of $Q_{0}$ will be contained in at least one of the $\d\Omega_{S}$.

We have to be more careful than this, however, since we could add two CADs $\Omega_{S}$ and $\Omega_{S'}$ corresponding to stopping-time regions $S$ and $S'$ for which $Q(S)$ and $Q(S')$ are adjacent, and if $\d\Omega_{S}$ and $\d\Omega_{S'}$ contain all of $Q(S)$ and $Q(S')$, this could cause a pinch in the domain $\Omega_{0}$. To remedy this, we also remove a small neighborhood of the boundaries of the $Q(S)$ and remove cubes from our stopping-times that fall into these gaps. In this way, the points in $\d\Omega_{S}$ and $\d\Omega_{S'}$ that are close to $\d\Omega$ will be far enough away from each other. \\

We proceed with the proof. Let $\delta>0$. Fix $S_{0}\in \cF$ and set $Q_{0}=Q(S_{0})$. Let $N$ be large (we will choose it shortly). Let 
\[
\cT=\cB\cup\{Q(S):S\in \cF\}\cup\bigcup_{S\in \cF}m(S).
\]
and
\[
E_{0}=\{x\in Q_{0}:\sum_{Q\in \cT \atop Q\subseteq Q_{0}} \one_{Q}\geq N\}.
\]
Then
\[
\sigma(E_{0})
\leq \frac{1}{N}\int \sum_{Q\in \cT \atop Q\subseteq Q_{0}}\one_{Q}d\sigma
\leq \frac{1}{N}\sum_{Q\in \cT \atop Q\subseteq Q_{0}}\sigma(Q)
\lec \frac{\sigma(Q_{0})}{N}.
\]
We fix $N$ so that 
\[
\sigma(E_{0})<\frac{\delta}{2} \sigma(Q_{0}).
\]

Recall from Theorem \ref{t:Christ} that, for $\tau>0$,

\[
\sigma(Q\backslash (1-\tau)Q)\lec \tau^{\alpha}\sigma(Q). 
\]

%
%
%
 Let 
 
\[
E_{1}=\bigcup_{Q\in \cT \atop Q\subseteq Q_{0}}  Q\backslash (1-\tau)Q.
\]
Then 
\[
\sigma(E_{1})
\leq \sum_{S\in \cF\atop Q(S)\subseteq Q_{0}} \sigma(Q(S)\backslash (1-\tau))Q(S))
\lec \sum_{S\in \cF\atop Q(S)\subseteq Q_{0}}  \tau^{\alpha} \sigma(Q(S))
\lec \tau^{\alpha} \sigma(Q_{0}).
\]
So for $\tau$ small enough
\[
\sigma(E_{1})<\frac{\delta}{2}\sigma(Q_{0}).\]
Let 
\[
G=Q_{0}\backslash (E_{0}\cup E_{1}).
\]
Then
\begin{equation}
\label{e:outGsmall}
\sigma(Q_{0}\backslash G)<\delta \sigma(Q_{0}). 
\end{equation}

For $S\in \cF$ with $Q(S)\subseteq Q_{0}$, let $Q_{j}$ be the cubes that have a sibling not intersecting $G$ and set 
\[
\tilde{S}=\{Q\in S:Q\not\subseteq  Q_{j} \mbox{ for some }j\}
\]
and 
\[
\tilde{\cF}=\{\tilde{S}:S\in \cF,\;\; Q(S)\subseteq Q_{0}\}.
\]
\begin{remark}
\label{r:remark}
We make a few remarks:
\begin{enumerate}
\item Note that for $\delta$ small, there is $c(\delta)\in (0,1)$ so that any $Q\subseteq Q_{0}$ with $\ell(Q)\geq \ell(Q_{0})$ intersects $G$.
\item Each $S\in \tilde{\cF}$ are stopping-time regions. 
\item Each $x\in G$ is contained in infinitely many cubes from one $S\in \tilde{\cF}$, and is only contained in $N$ many $Q(S)$ with $S\in \tilde{\cF}$. 
\item Finally, 
\begin{multline}\label{e:taudist}
\mbox{if} \;\; Q, R\in \cT \;\;\mbox{and} \;\; Q\cap R=\emptyset , \mbox{ then }  \\
\dist(Q\cap G,R\cap G)>\tau\max\{\ell(Q),\ell(R)\}.
\end{multline}

\end{enumerate}
\end{remark}

\def\tS{\tilde{S}}
For $n\geq 0$, let $\cF_{0}=\{S_{0}\}$ (recall $S_{0}$ is the stopping time region with $Q(S_{0})=Q_{0}$) and $\cF_{n}$ denote those $\tS\in \tilde{\cF}$ that are properly contained in $n$ many cubes of the form $Q(\tilde{S}')\subseteq Q_{0}$ where $\tilde{S}'\in \cF$ and $Q(\tilde{S})\cap G\neq\emptyset$. %
So in particular, $\cF_{0}=\{\tS_{0}\}$ and $\cF_{n}$ are those stopping-times $\tS\in \tilde{\cF}$ that are properly contained in $n$ other $Q(\tS)\subseteq Q_{0}$ with $\tS\in \tilde{\cF}$. 

Let $S\in \tilde{\cF}$, and let $\hat{S}\in \cF$ be so that $\tilde{\hat{S}}=S$. For $x\in \R^{d+1}$, define
\[
d_{S}(x)=\inf_{Q\in S}(\dist(x,Q)+\ell(Q)).\] 

Note that $d_{S}$ is $1$-Lipschitz: indeed, if $x,y\in \R^{d+1}$ and $Q\in S$, then
\[
d_{S}(x)\leq \dist(x,Q)+\ell(Q)\leq |x-y|+\dist(y,Q)+\ell(Q),\]
and infimizing over all $Q\in S$ gives $d_{S}(x)\leq |x-y|+d_{S}(y)$, which proves the claim.

Let $\Gamma_{S}:= \Gamma_{\hat{S}}$, where $\Gamma_{\hat{S}}$ is from Lemma \ref{l:HMM}. Assume $\Gamma_{S}$ is a graph over a $d$-plane $P_{S}$ of a function $f_{S}$. Let $e_{S}$ be a normal unit vector to $P_{S}$, $\pi_{S}$ the projection into $P_{S}$, and for $x\in \R^{d+1}$, define
\[
F_{S}(x)=\pi_{S}(x)+ f_{S}(\pi_{S}(x))e_{S}.
\]
Then each $x\in \Gamma_{S}$ can be written as 
\[
x=F_{S}(x).\]

Set
\[
U_{S}^{+} = \{x+te_{S}\in KB_{Q(S)}:x\in P_{S},\;\;   t>f_{{S}}(x)+d_{S}(F_{S}(x))\}
\]
and
\[
U_{S}^{-} = \{x+te_{S}\in KB_{Q(S)}:x\in P_{S},\;\;   t<f_{{S}}(x)-d_{S}(F_{S}(x))\}
\]
Note that these two domains are Lipschitz domains. See Figure \ref{f:OmegaS1}.

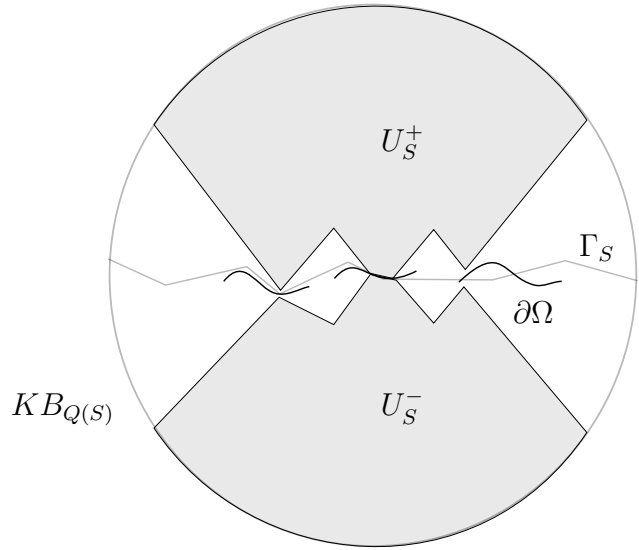
\begin{figure}[!ht]
\includegraphics[width=200pt]{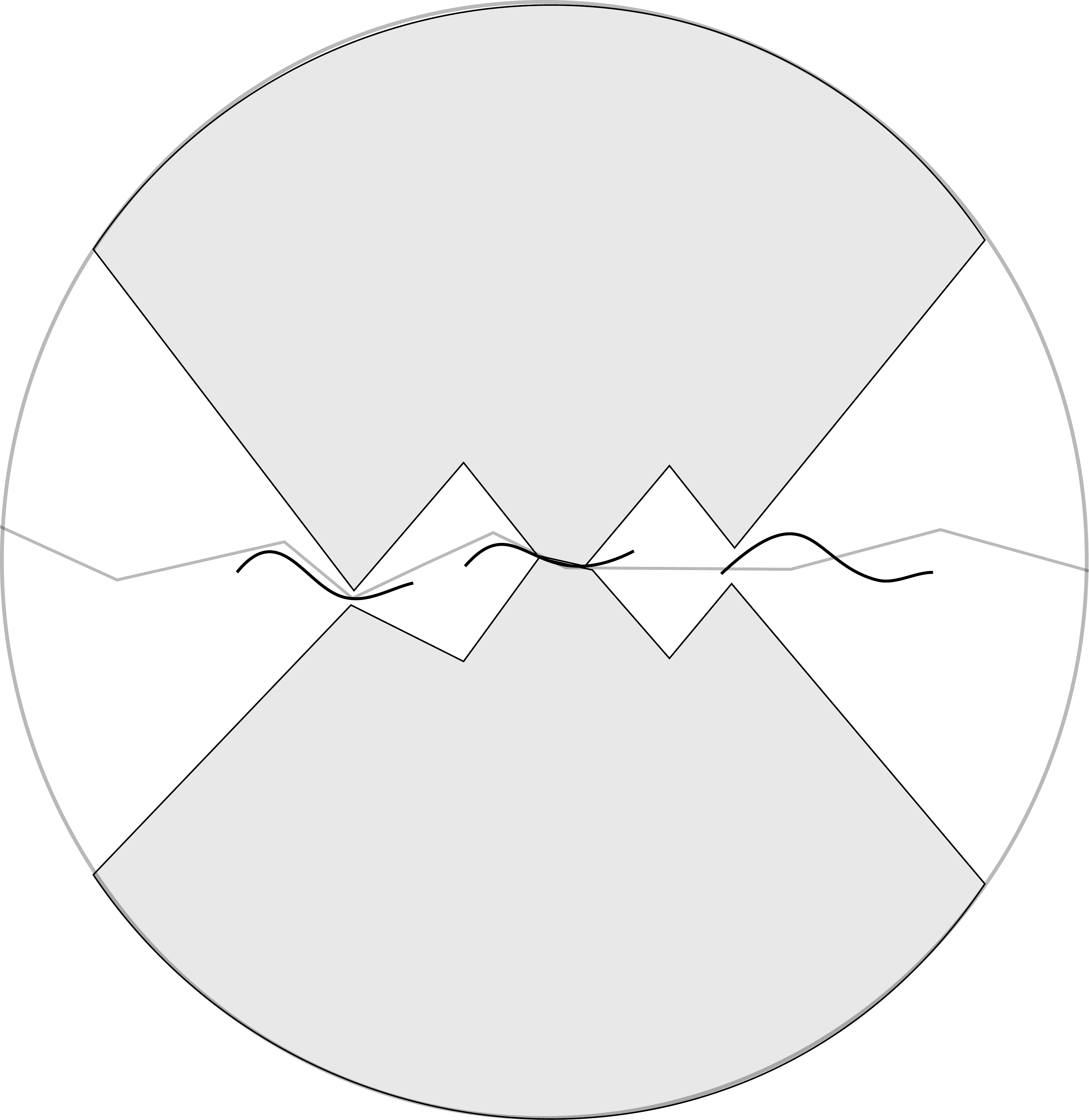}
\begin{picture}(0,0)(200,0)
\put(100,150){$U_{S}^{+}$}
\put(100,50){$U_{S}^{-}$}
\put(175,110){$\Gamma_{S}$}
\put(150,85){$\d\Omega$}
\put(-40,50){$KB_{Q(S)}$}
\end{picture}
\caption{The domains $U_{S}^{\pm}$ above and below $\d\Omega$ in $KB_{Q(S)}$.}
\label{f:OmegaS1}
\end{figure}

Let $\cW$ denote the Whitney cubes in $\Omega$ and define
\[
\cW_{S}^{\pm}=\{Q\in \cW :  Q\cap \cnj{U_{S}^{\pm}}\neq\emptyset\}
\]
and for $\lambda>1$ close to one, let
\[
\Omega_{S}^{\pm}=\bigcup_{Q\in \cW_{S}^{\pm}} \lambda Q.
\]
See Figure \ref{f:OmegaS2}. It is not hard to show that $\Omega_{S}^{\pm}$ are CADs. 

\begin{figure}[!ht]
\includegraphics[width=300pt]{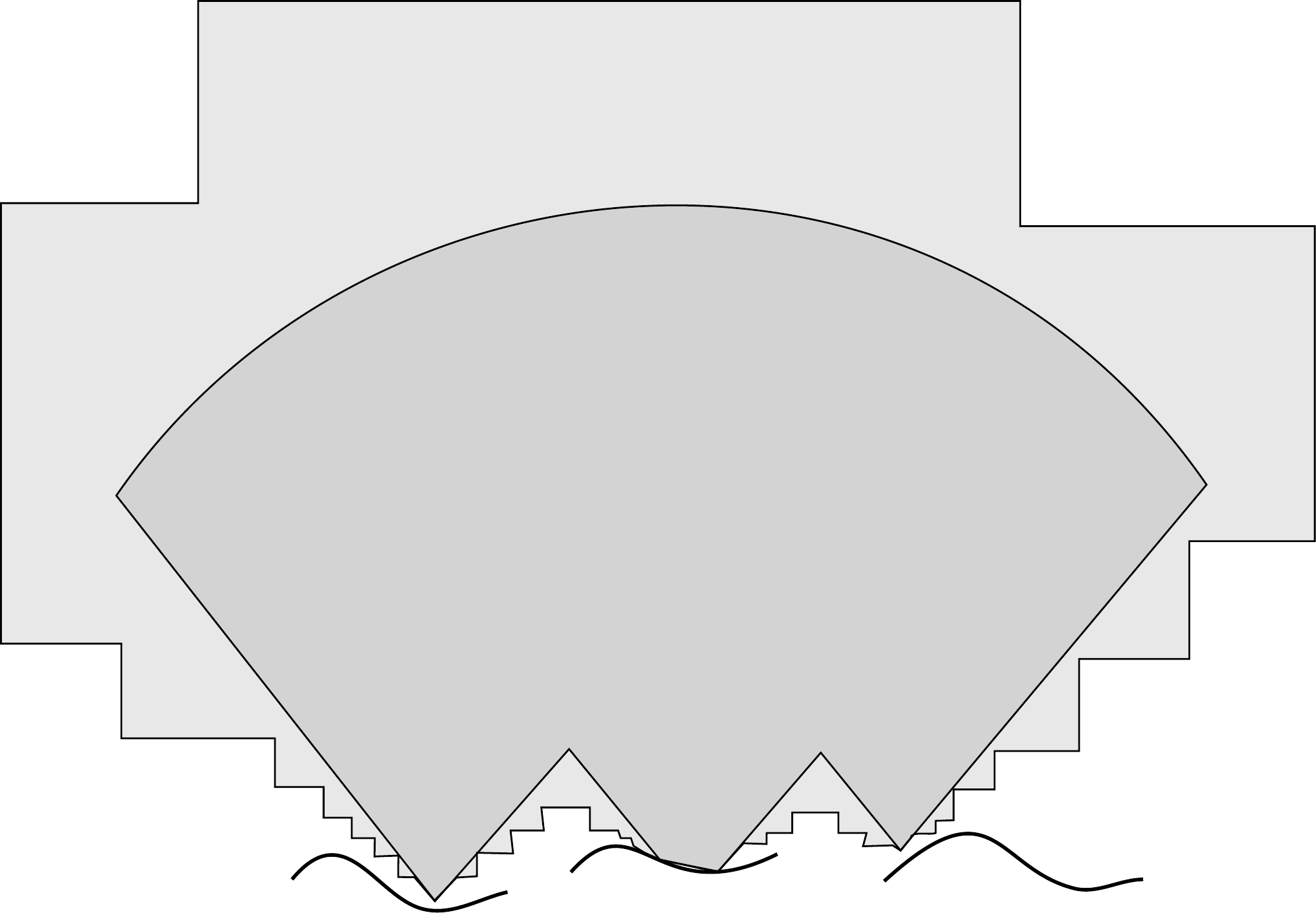}
\begin{picture}(0,0)(350,0)
\put(200,100){$U_{S}^{+}$}
\put(200,175){$\Omega_{S}^{+}$}
\put(80,0){$\d\Omega$}
\end{picture}
\caption{The domain $\Omega_{S}^{+}$ is obtained by fattening $U_{S}^{+}$ with fattened Whitney cubes.  }
\label{f:OmegaS2}
\end{figure}
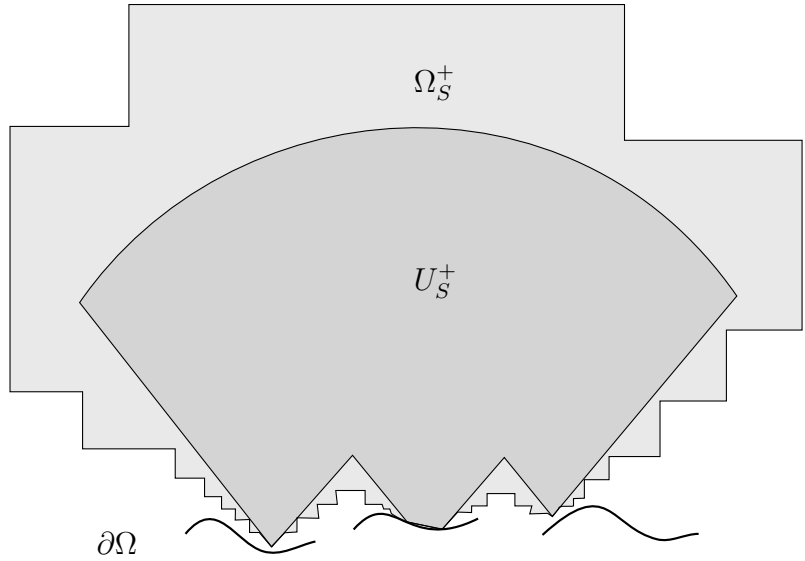

For $Q \in m(S)$, recall from Theorem \ref{t:Christ} that $\zeta_{Q}$ is the center of $Q$. Let $\gamma_{Q}^{\pm}$ be the center of the cube $I_{Q}^{\pm} \in \cW_{S}^{\pm}$ for which 
\[
F_{S}(\pi_{S}(\zeta_{Q}))\pm 
 d_{S}(F_{S}(\zeta_{Q}))e_{S} 
\in  I_{Q}^{\pm}.
\]
Note that $F_{S}(\pi_{S}(\zeta_{Q}))\in \Gamma_{S}\cap \frac{K}{2}B_{Q(S)}$ by definition and for $K\gg 1$, so the vector above is in $U_{S}^{\pm}\cap \frac{K}{2}B_{Q(S)}$, hence $I_{Q}^{\pm}$ is well defined.

Note that by the interior corkscrew property, at least one of the domains $\Omega_{S}$ must be in $\Omega$, and both could be. We will pick exactly one of them for each $S\in \cF$ as follows. 

Pick a corkscrew point for $B_{Q_{0}}$. For $\eta>0$ small enough, we can guarantee it is also a corkscrew point for either $\Omega_{\tilde{S_{0}}}^{\pm}$, say it is $\Omega_{\tilde{S_{0}}}^{+}$ and set
\[
\Omega_{{S_{0}}}=\Omega_{{S_{0}}}^{+} \;\;\; \mbox{ and }\;\;\; U_{S_{0}}=U_{S_{0}}^{+}.
\]
Let $n\geq 0$ and suppose we have defined $\Omega_{S}$ for all $S\in {\cF}_{n}$. Let $S\in {\cF}_{n+1}$. Note that by assumption, at most $N$ ancestors of $Q(S)$ are not in some other stopping time $S'\in \tilde{S}$. In particular, if $S'\in \cF_{n}$ is such that $Q(S)\subseteq Q(S')$, there is $\hat{Q}(S)\in m(S')$ so that $\hat{Q}(S)\supseteq Q(S)$ and $\ell(\hat{Q}(S))/\ell(Q(S))\lec_{N} 1$.

By semi-uniformity, there is a Harnack chain from $x_{S}$ (which we define to be whichever of $\gamma_{\hat{Q}(S)}^{\pm}$ is in $\Omega_{S'}$)  to a corkscrew in $B_{Q(S)}$. See Figure \ref{f:omegaS3} below. However, by \eqref{e:closegraph} and since $\Gamma_{S}$ is $\eta$-Lipschitz, for $\eta>0$ small enough, we know this corkscrew point is also a corkscrew point for either $\Omega_{S}^{\pm}$, we suppose it is $\Omega_{S}^{+}$ and call the corkscrew point $x(S)$. 


Let $T_{S}'$ be the union of the balls in this Harnack chain and let 
\[
T_{S}=\bigcup_{Q\in \cW \atop Q\cap T_{S}'\neq\emptyset} \lambda Q. 
\]

\begin{figure}[!ht]
\includegraphics[width=300pt]{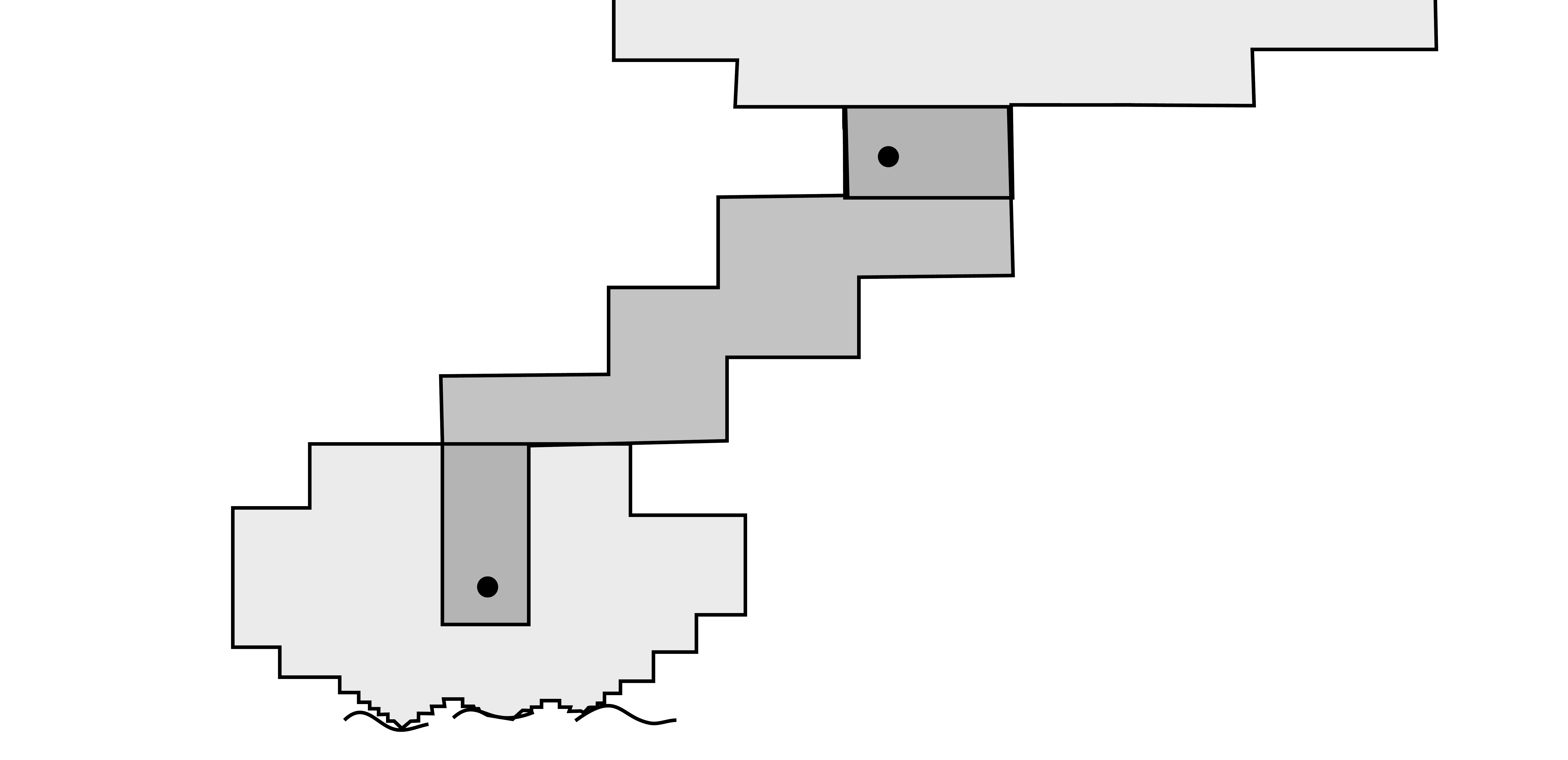}
\begin{picture}(0,0)(350,0)
\put(100,25){$\Omega_{S}^{+}$}
\put(190,80){$T_{S}$}
\put(260,130){$\Omega_{S'}$}
\put(220,110){$x_{S}$}
\end{picture}
\caption{This shows how $\Omega_{S}$ constructed by adding to $\Omega_{S}^{+}$ a path of cubes $T_{S}$ from $\Omega_{S}^{+}$ to the domain $\Omega_{S'}$ above it.}
\label{f:omegaS3}
\end{figure}
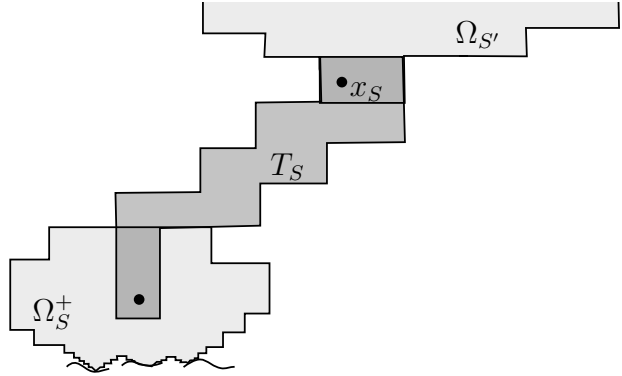

We now set
\[
\Omega_{S}= \Omega_{S}^{+}\cup T_{S}.
\]
Note that
\[
x_{S}\in \Omega_{S'}\cap \Omega_{S}\]
and is a corkscrew point for each of these domains in the sense that 
\[
\delta_{\Omega_{S}}(x_{S})\sim 
\delta_{\Omega_{S'}}(x_{S})
\sim \delta_{\Omega}(x_{S})\sim \ell(Q(S)).\] 
It is also not hard to show that each $\Omega_{S}$ is a CAD. 

Now we define
\[
\Omega_{0}=\bigcup_{S\in \tilde{\cF}} \Omega_{S}.
\]
Note that $\diam \Omega_{0}\sim \ell(Q_{0})$, and it is not hard to see that 
\[
\d\Omega_{0}\cap \d\Omega = \cnj{G}.
\]
Also, let $\cW_{0}$ be all cubes $Q\in \cW$ that intersect $U_{S}\cup T_{S}$ for some $S\in \tilde{\cF}$, then we have by definition that
\[
\Omega_{0}=\bigcup_{Q\in \cW_{0}} \lambda Q.
\]

\begin{lemma}
For $S\in \tilde{\cF}$,
\begin{equation} 
\label{e:closegraph2}
\dist(x,\d\Omega)\lec \eta d_{S}(x) \mbox{ for all }x\in \Gamma_{S}\cap K^2B_{Q(S)}
\end{equation}
and
\begin{equation} 
\label{e:closegraph3}
\dist(x,\Gamma_{S})\lec \eta d_{S}(x) \mbox{ for all }x\in \d\Omega \cap K^2B_{Q(S)}
\end{equation}
\end{lemma}

\begin{proof}
First we'll prove \eqref{e:closegraph2}, the proof of \eqref{e:closegraph3} is identical. Let $x\in \Gamma_{S}\cap K^2B_{Q(S)}$ and $Q'\in S$ be such that 
\[
\ell(Q')+\dist(x,Q') \leq 2d_{S}(x).\]
We divide into a few cases.
\begin{enumerate}
\item If $\ell(Q')>\dist(x,Q')$, then $x\in K^2B_{Q}$ for $K\gg 1$, and so \eqref{e:closegraph} implies
\[
\dist(x,\d\Omega) <\eta\ell(Q')\leq 2\eta d_{S}(x).\]
\item Now suppose $\ell(Q')\leq \dist(x,Q')$. Replace $Q'$ with its largest ancestor $Q\in S$ for which $\ell(Q)\leq \dist(x,Q)$, so we still have 
\[
\ell(Q)+\dist(x,Q)\leq 2\dist(x,Q)\leq 2\dist(x,Q')\leq 4d_{S}(x).\]
\begin{enumerate}[(a)]
\item If $Q\neq Q(S)$, then $\ell(Q)\sim \dist(x,Q)$, so for $K$ large enough, $x\in K^2B_{Q}$, and so \eqref{e:closegraph} implies
\[
\dist(x,\d\Omega)<\eta \ell(Q)\lec \eta d_{S}(x).
\]
\item If $Q=Q(S)$, then 
\[
\ell(Q)\leq \dist(x,Q)\leq \diam K^2B_{Q(S)} = 2K^2\ell(Q(S))=2K^2\ell(Q).
\]
Thus, since $x\in K^2B_{Q(S)}$, \eqref{e:closegraph} again implies
\[
\dist(x,\d\Omega)<\eta \ell(Q(S))=\eta\ell(Q)\leq 2\eta d_{S}(x).
\]
\end{enumerate}\end{enumerate}

\end{proof}

\begin{lemma}
\[
\dist(x,G)\lec  \dist(x,\d\Omega)\mbox{ for all }x\in \Omega_{0}.
\]
In particular, this follows from
\begin{equation}
\label{e:gs}
d_{S}(x)\sim \dist(x,G\cap Q(S))\sim \dist(x,\d\Omega) \mbox{ for all }x\in \Omega_{S}.
\end{equation}

\end{lemma}

\begin{proof}

Suppose $x\in \Omega_{S}$, then $x\in \lambda I $ for some $I\in \cW_{0}$ so that $I\cap U_{S}\neq\emptyset$. Let $y\in I\cap U_{S}$ and $y'=F_{{S}}(\pi_{S}(y))\in \Gamma_{S}$. \\

We first claim that 
\begin{equation}
\label{e:disty-y'}
\dist(y,\d\Omega)\sim |y-y'|.
\end{equation}
Note that since $y\in U_{S}^{+}$,
\begin{align}
|y-y'|
& \geq (y-y')\cdot e_{S}
=y\cdot e_{S} - F_{S}(\pi_{S}(y))\cdot e_{S} \notag \\
& >f_{S}(\pi_{S}(y))+d_{S}(F_{S}(y))-f_{S}(\pi_{S}(y))
=d_{S}(F_{S}(y)).
\label{e:y-y'<ds}
\end{align}
This and the fact that $\Gamma_{{S}}$ is a Lipschitz graph imply
\begin{equation}
\label{e:y-y'<gammas}
d_{S}(y') \leq  |y-y'|\sim \dist(y,\Gamma_{{S}}).
\end{equation}
%
%
By \eqref{e:closegraph} we have $\dist(\zeta_{Q(S)},\Gamma_{S})<\eta \ell(Q(S))$, and since $y\in KB_{Q(S)}$,
\begin{equation}
\label{e:y-zeta<K+1}
|y-y'|\sim \dist(y,\Gamma_{S})
\leq |y-\zeta_{Q(S)}|+\eta \ell(Q(S))
< (K+1)\ell(Q(S)),
\end{equation}
and so
\begin{equation}
\label{e:y'inBS}
|y'-\zeta_{Q(S)}|\leq |y'-y|+|y-\zeta_{Q(S)}|
\stackrel{\eqref{e:y-zeta<K+1}}{\leq}
 \lec K\ell(Q(S)).
\end{equation}
Thus, for some constant $C$ and for $K$ large enough (depending on $C$), 
\[
y'\in \Gamma_{S}\cap CKB_{Q(S)}\subseteq K^{2}B_{Q(S)}.\] 
Hence, for $\eta\ll1$,  \eqref{e:closegraph2} implies
\[
\dist(y,\d\Omega)\leq |y-y'|+\dist(y',\d\Omega)
<|y-y'|+C\eta d_{S}(y')
\leq 2|y-y'|.
\]
So it suffices to show $|y-y'|\lec \dist(y,\d\Omega)$. 

If $B(y,|y-y'|)\cap \d\Omega=\emptyset$, then the inequality follows immediately, so assume instead that there is 
\begin{equation}
\label{e:zeta}
\zeta\in B(y,|y-y'|)\cap \d\Omega.
\end{equation} Since $y,KB_{Q(S)}$ and $y'\in CKB_{Q(S)}$, we have 
\[
\zeta\in (C+1)KB_{Q(S)}\subseteq K^{2}B_{Q(S)}\]
 so by \eqref{e:closegraph3},
\[
\dist(\zeta,\Gamma_{{S}})
< \eta d_{S}(\zeta)
\leq \eta (d_{S}(y')+|y'-\zeta|)
\stackrel{\eqref{e:zeta}}{\leq} \eta( d_{S}(y')+2|y'-y|)
\stackrel{\eqref{e:y-y'<ds}}{\lec} \eta |y-y'|.
\]
Let $z=F_{S}(\pi_{S}(\zeta))$. 
Since $\Gamma_{S}$ is a $\eta$-Lipschitz graph, the previous inequality implies $|\zeta-z|\lec \eta |y-y'|$ and so
\begin{align*}
e_{S}(y'-\zeta)
& \leq e_{S}\cdot (y'-z) + C\eta |y-y'|
 \leq \eta |\pi_{S}(y')-\pi_{S}(z)|+ C\eta |y-y'|\\
& \leq \eta |y'-\zeta|+ C\eta |y-y'|
\lec \eta|y-y'|.
\end{align*}
Thus, for all $\zeta\in B(y,|y-y'|)\cap \d\Omega$, 
\[
|y-\zeta|
\geq e_{S}\cdot (y-\zeta)
\geq e_{S}\cdot (y-y')+ e_{S}\cdot (y'-\zeta)
\geq |y-y'|-C\eta |y-y'|
\gec |y-y'|.
\]
Infimizing over all $\zeta\in B(y,|y-y'|)\cap \d\Omega)$, we get
\[
\dist(y,\d\Omega)
=\dist(y,B(y,|y-y'|)\cap \d\Omega)
\gec |y-y'|.
\]
And this proves \eqref{e:disty-y'}. 
%

Since $I$ is a Whitney cube, $\dist(z,\d\Omega)\sim \diam(I)$ for all $z\in \lambda I$, and so
\[
d_{S}(x)
\leq d_{S}(y')+|y'-y|+|y-x|
\stackrel{\eqref{e:y-y'<ds}}{\leq} 2|y'-y|+\diam I
\stackrel{\eqref{e:disty-y'}}{\lec} \dist(y,\d\Omega)
\]
and so
\[
d_{S}(x) \lec \dist(y,\d\Omega)\sim \dist(x,\d\Omega)
\leq \dist(x,G\cap Q(S)).\]
Moreover, if $T\in S$ is such that $\dist(x,T)+\ell(T)<2d_{S}(x)$, then $T\cap G\neq\emptyset$, so we have that
\[
\dist(x,G\cap Q(S))
\leq \dist(x,T)+\diam T\lec  d_{S}(x).
\]

\end{proof}

\begin{lemma}
\label{l:omegacorkscrews}
The domain $\Omega_{0}$ has the interior corkscrew property. 
\end{lemma}

\begin{proof}

Let $z\in \d\Omega_{0}$ and $0<r<\diam \Omega_{0}$. Then 
\[
d_{S_{0}}(x)\leq \ell(Q_{0})+\diam \Omega_{0}\lec \ell(Q_{0}),\]
and so $d_{S_{0}}(x)<M\ell(Q_{0})$ for some number $M$ depending on the semi-uniformity of $\Omega$. 

Let $x\in \d\Omega_{0}$ and $0<r<\diam \d\Omega_{0}$. We split into several cases. \\

\begin{enumerate}
\item If $x\in \d\lambda I$ for some $I\in \cW_{0}$ and $r\leq M  \ell(I)$, then we can pick any ball inside $B(x,r)\cap \lambda I$ of radius comparable to $ r$. \\

\item  Suppose $x\in \d\lambda I$ for some $I\in \cW_{0}$ and $M\ell(I)<r<\diam \d\Omega_{0}$. Note that by \eqref{e:gs},
\[
\dist(x,G)
\lec \dist(x,\d\Omega)
\lec \ell(I)\leq M^{-1} r,
\]
hence for $M$ large enough, we can find $y\in G\cap B(x,\frac{r}{2})$. Thus, we just need to find an interior corkscrew for $B(y,r/2)$, since it will then also be one for $B(x,r)$.

Let $S\in \tilde{\cF}$ be so that $y\in Q(S)$ and $\ell(Q(S))$ is minimal such that $\ell(Q(S))> r/M^{2}$. 

\begin{enumerate}[(a)]
\item[(2.a)] If $\ell(Q(S))< r/M$, then for $M$ large enough, $\Omega_{S}\subseteq B(x,r)$ and $x_{S}$ is a corkscrew point for $B(x,r)$.
\item[(2.b)] If $\ell(Q(S))\geq r/M$, then recall from Remark \ref{r:remark} that since $y\in G$, $y$ is contained in infinitely many cubes from some stopping time in $\tilde{F}$. 
\begin{enumerate}[(i)]
\item[(2.b.i)] If this stopping time is not $S$, then there is  $S'\tilde{F}$ a maximal stopping time for which $Q(S')\subsetneq Q(S)$ and $y\in Q(S')$. Then $\ell(Q(S'))\leq r/M^{2}$ by the minimality of $S'$, and so $\Omega_{S'}\subseteq B(y,\frac{r}{4M})$ for $M$ large enough. Since $\Omega_{S}$ is chord-arc, it is uniform and there is a cigar curve $\gamma$ of bounded turning between $x_{S'}$ and $x_{S}$. Let $z\in \d B(y,\frac{r}{2M})\cap \gamma$. Then
\begin{align*}
\delta_{\Omega}(z)
& \geq \delta_{\Omega_{S}}(z)
\gec \min \{|z-x_{S'}|,|z-x_{S}|\}\\
& \geq \min\ck{\frac{r}{2N}-\frac{r}{4M},\frac{r}{M}-\frac{r}{2M}}\gec \frac{r}{M}.
\end{align*}
Thus, $z$ is a corkscrew point for $B(y,r/2)$ with constant depending on $M$.
\item[(2.b.ii)] If it is $S$, then $y\in \d \Omega_{S}$, and we can connect $y$ to $x_{S}$ directly with a cigar curve of bounded turning in $\Omega_{S}$ and then the proof is just as in the previous case. 
\end{enumerate}
\end{enumerate}

\item If $x\in \cnj{G}=\d\Omega\cap \d\Omega_{0}$, then we repeat the argument in Case 2 with $x$ in place of $y$. 
\end{enumerate}
\end{proof}

\begin{lemma}
\label{l:omegauniform}
The domain $\Omega_{0}$ is uniform. 
\end{lemma}

\begin{proof}
Let $x,y\in \Omega_{0}$, $\ve = \min\{\delta_{\Omega_{0}}(x),\delta_{\Omega_{0}}(y)\}$, and $\Lambda = |x-y|$. By Lemmas \ref{l:ahmnt} and \ref{l:omegacorkscrews}, it suffices to show that there is a Harnack chain between every such $x$ and $y$ of length at most $C\log\frac{\Lambda}{\ve}+1$. As in \cite{HM14}, we will first make some reductions. \\

{\bf Case 1.} First, we may assume that any cubes $Q,R\in \cW_{0}$ containing $x$ and $y$ respectively satisfy 
\begin{equation}
\label{e:QRfar}
\dist(\lambda Q, \lambda R)\gec \max\{\ell(Q),\ell(R)\}.
\end{equation}
Since otherwise $Q$ and $R$ must be adjacent and $\lambda Q\cup \lambda R$ forms an NTA domain and it is easy to find a Harnack chain in $\Omega_{0}$ between $x$ and $y$ 

{\bf Case 2.} Next, we claim it suffices to assume $x$ and $y$ are centers of cubes in $\cW_{0}$. 

Indeed, if $x\in \lambda Q$ and $\delta_{\Omega_{0}}(x)\gec \ell(Q)$, then it is not hard to find a short Harnack chain from $x$ to $x_{Q}$. If $\delta_{\Omega_{0}}<\ell(Q)$,  then we may connect $x$ to a point $x'\in B(x,\delta_{\Omega_{0}}(x))\cap \lambda Q$ of distance at least a constant times $\delta_{\Omega_{0}}(x)$ from $\d\lambda Q$, and thus there is a Harnack chain from that point to $x_{Q}$ of length depending on $\log \diam \lambda Q / \delta_{\Omega_{0}}(x) \lec \log  \Lambda/\ve $. Similarly, if $y\in \lambda R$ for some $R\in \cW_{0}$, we can connect $y$ to $x_{R}$ by a chain of length depending on $\Lambda$. Then, by assumption, we can connect $x_{Q}$ and $x_{R}$ by a chain of length at most a constant times
\[
\log \frac{|x_Q-x_{R}|}{\min\{\delta_{\Omega_{0}}(x_{Q}),\delta_{\Omega_{0}}(x_{R})\}}+1
\lec \log \frac{\Lambda}{\ve}+1
\]
where in the last inequality we used the fact that \eqref{e:QRfar} implies $|x_{Q}-x_{R}|\sim |x-y|$. Combining these two chains together gives the desired chain between $x$ and $y$. 

{\bf Case 3.} We now assume $x=x_{Q}$ and $y=x_{R}$ for some cubes $Q,R\in \cW_{0}$ satisfying \eqref{e:QRfar}. \\

{\bf Case 3.a.} Suppose $x_{Q},x_{R}\in \Omega_{S}$. Then since $\Omega_{S}$ is a CAD, it is not hard to find the desired Harnack chain between $x_{Q}$ and $x_{R}$.\\

{\bf Case 3.b.}  Suppose $x_{Q}\in  \Omega_{S_{1}}$ and $x_{R}\in  \Omega_{S_{n}}\backslash \Omega_{S_{n-1}}$ for some $S_{1},S_{n}\in \tilde{\cF}$ where $Q(S_{1})\subseteq Q(S_{n})$ and there are $S_{2},...,S_{n-1}$ so that $Q(S_{i})\subseteq Q(S_{i+1})$. 

Note that since $Q\in \cW_{S}$ and $\lambda Q\subseteq \Omega_{S}\subseteq \Omega_{0}$,
\[
\ell(Q)\lec \delta_{\Omega_{S}}(x_{Q})
\leq  \delta_{\Omega_{0}}(x_{Q})
\leq \delta_{\Omega}(x_{Q})
\sim \ell(Q).
\]

Since $\Omega_{S_{1}}$ is a CAD, there is a Harnack chain from $x_{Q}$ to $x_{S_{1}}$ of length
\[
\lec 
\log\frac{\ell(Q(S_{1}))}{\delta_{\Omega_{S_{1}}}(x_{Q})}+1
\lec \log\frac{\ell(Q(S_{1}))}{\min\{\ell(Q),\ell(R)\}}+1.
\]
For each $i$, there is a Harnack chain in $\Omega_{S_{i+1}}$ from $ x_{S_{i}}$ to $x_{S_{i+1}}$ of length
\[
\lec \log\frac{ |x_{S_{i}}-x_{S_{i+1}}|}{\min\{\delta_{\Omega_{S_{i+1}}}(x_{S_{i}}),\delta_{\Omega_{S_{i+1}}}(x_{S_{i+1}})\}}+1
\lec \log\frac{\ell(Q(S_{i+1}))}{ \ell(Q(S_{i}))}+1.
\]
Finally, we can also show there is a Harnack chain in $\Omega_{S_{n}}$ between $x_{S_{n-1}}$ to $x_{R}$ of length
\[
\lec \log\frac{|x_{R}-x_{S_{n-1}}|}{\min\{\delta_{\Omega_{S_{n}}}(x_{S_{n-1}}),\delta_{\Omega_{S_{n}}}(x_{R})\}}+1
\leq  \log\frac{|x_{R}-x_{Q}| + \ell(Q(S_{n-1}))}{\min\{\ell(R),\ell((Q(S_{n-1}))\}}+1.
\]
Here, we used the fact that $|x_{Q}-x_{S_{n-1}}|\lec \ell(Q(S_{n-1}))$. Adding up these inequalities, we get that there is a Harnack chain from $x_{R}$ to $x_{Q}$ of length no more than
\begin{multline*}
\lec \log\frac{\ell(Q(S_{1}))}{\min\{\ell(Q),\ell(R)\}}+
\sum_{i=1}^{n-2}\ps{ \log\frac{\ell(Q(S_{i+1}))}{ \ell(Q(S_{i}))}+1}\\
+\log\frac{|x_{R}-x_{Q}| + \ell(Q(S_{n-1}))}{\min\{\ell(R),\ell((Q(S_{n-1}))\}}\\
\lec \log\frac{\ell(Q(S_{n-1}))}{\min\{\ell(Q),\ell(R)\}}+\log\frac{|x_{R}-x_{Q}| + \ell(Q(S_{n-1}))}{\min\{\ell(R),\ell(Q(S_{n-1}))\}}+N
\end{multline*}

Our goal is to show
\begin{equation}
\label{e:xqxr>qs}
\ell(Q(S_{n-1}))\lec |x_{Q}-x_{R}|
\end{equation}
in which case the total length of our combined chain will be at most 
\begin{equation}
\label{e:finallength}
\lec \log \frac{|x_{R}-x_{Q}|}{\min\{\ell(Q),\ell(R)\}}+N
\lec \log \frac{\Lambda}{\ve}+1
\end{equation}

{\bf Case 3.b.i.} If $\max\{\ell(Q),\ell(R)\} \geq \tau^{2}  \ell(Q(S_{n-1}))$, then by \eqref{e:QRfar}, we also have that $|x_{Q}-x_{R}|\gec \ell(Q(S_{n-1}))$, which proves \eqref{e:xqxr>qs}. 
%

{\bf Case 3.b.ii.} Now suppose $\max\{\ell(Q),\ell(R)\}<\tau^{2}\ell(Q(S_{n-1}))<\tau^{2}\ell(Q(S_{n}))$.

By \eqref{e:gs}, we may find $\xi_{R}\in Q(S_{n})\cap G$ so that
\[
|\xi_{R}-x_{R}|
\lec \dist(x_{R},\d\Omega) 
\sim \ell(R)< \tau^{2} \ell(Q(S_{n-1})).\]
Similarly, we may find $\xi_{Q}\in Q(S_{1})\cap G\subseteq Q(S_{n-1})\cap G$ so that 
\[
|\xi_{Q}-x_{Q}|\lec \tau^{2} \ell(Q(S_{n-1})).\]

Suppose $\xi_{R}\in Q(S_{n-1})$. Let $T_{0}\in m(S_{n})$ contain $Q(S_{n-1})$. If $T\in S_{n}$ is disjoint from $T_{0}$, then by \eqref{e:taudist}
\begin{multline*}
\dist(x_{R},T)+\ell(T)
\geq \dist(\xi_{R},T)-|\xi_{R}-x_{R}|
\geq \tau \ell(T_{0})-C\tau^{2}\ell(Q(S_{n-1}))
\\ \gec \tau \ell(Q(S_{n-1}))
\end{multline*}
and if $T\in S_{n}$ contains $T_{0}$, clearly 
\[
\dist(x_{R},T)+\ell(T)
\geq \ell(T)\geq \ell(Q(S_{n-1})).\]
Thus, infimizing over all $T\in S_{n}$, we get
\[
\tau^{2} \ell(Q(S_{n-1}))>\ell(R) \stackrel{\eqref{e:gs}}{\sim} d_{S_{n}}(x_{R})\gec \tau \ell(Q(S_{n-1})),\]
which is a contradiction for $\tau$ small enough. Thus, $\xi_{R}\in Q(S_{n})\backslash Q(S_{n-1})$, so $|\xi_{Q}-\xi_{R}|\geq \tau\ell(Q(S_{n-1}))$ by \eqref{e:taudist} since both points are in $G$, thus
\begin{multline*}
|x_{Q}-x_{R}|
\geq |\xi_{Q}-\xi_{R}|-|\xi_{Q}-x_{Q}|-|\xi_{R}-x_{R}|
\gec \tau \ell(Q(S_{n-1}))-\tau^{2}\ell(Q(S_{n-1}))
\\ \gec \tau \ell(Q(S_{n-1}))
\end{multline*}
and thus \eqref{e:xqxr>qs} holds in this case as well.\\

{\bf Case 3.c.}  Finally, suppose $x_{Q}\in \Omega_{S}$ and $x_{R}\in \Omega_{S'}$ where $Q(S)\cap Q(S')=\emptyset$. Let $\tilde{S}$ be the minimal $\tilde{S}\in \tilde{\cF}$ for which $Q(S),Q(S')\subseteq Q(\tilde{S})$. Let $Q_{S}$ and $Q_{S'}$ be the minimal cubes in $\tilde{S}$ containing $Q(S)$ and $Q(S')$ respectively.

We first claim that 
\begin{equation}
\label{e:xqxrqsqs'claim}
|x_{Q}-x_{R}|\gec \tau^{2} \max\{\ell(Q_{S}),\ell(Q_{S'})\}.
\end{equation}

If 
\[
\max\{\ell(Q),\ell(R)\}\geq \tau^{2}\max\{\ell(Q_{S}),\ell(Q_{S'})\},\]
then this holds by \eqref{e:QRfar}, so assume 
\[
\max\{\ell(Q),\ell(R)\}< \tau^{2}\max\{\ell(Q_{S}),\ell(Q_{S'})\}.\]
 Let $\xi_{Q}\in G\cap Q(S)$ and $\xi_{R}\cap G\cap Q(S')$ be closest to $x_{Q}$ and $x_{R}$ respectively. Then by \eqref{e:taudist},
\begin{align*}
|x_{Q}-x_{R}|
& \geq |\xi_{Q}-\xi_{R}|-|\xi_{Q}-x_{Q}|-|\xi_{R}-x_{R}|\\
& \gec \tau (\ell(Q_{S})+\ell(Q_{S'})-\tau^{2}\max\{\ell(Q_{S}),\ell(Q_{S'})\}\\
& \gec \tau\max\{\ell(Q_{S}),\ell(Q_{S'})\}
\end{align*}
and this finishes the claim.

We'll first build a chain from $x_S$ to $x_{S'}$. If $\tilde{S}\in \cF_{n}$, let  $\check{S},\check{S}'\in \cF_{n+1}$ be so that $Q(\check{S})\subseteq Q_{S}$ and $Q(\check{S'})\subseteq Q_{S'}$ contain $Q(S)$ and $Q(S')$ respectively. Observe that $\ell(Q_{S})\sim_{N} \ell(Q(\check{S}))$ and $\ell(Q_{S'})\sim_{N} \ell(Q(\check{S}'))$.  By case 3.b, we know there are chains from $x_{S}$ to $x_{\check{S}}\in \Omega_{\check{S}}\cap \Omega_{\tilde{S}}$ and from $x_{S'}$ to $x_{\check{S'}}\in \Omega_{\check{S'}}\cap \Omega_{\tilde{S}}$ of total length at most 
\begin{equation}
\label{e:the-ups}
\log \frac{\ell(Q(\check{S}))}{\ell(Q(S))}+\log \frac{\ell(Q(\check{S'}))}{\ell(Q(S'))}+1
\lec \log \frac{\max\{\ell(Q(\check{S})),\ell(Q(\check{S'}))\}}{\min\{ \ell(Q),\ell(R)\}}+1
\end{equation}
Next, there is a Harnack chain from $x_{\check{S}}$ to $x_{\check{S}'}$ in $\Omega_{\tilde{S}}$ of length at most
\begin{multline*}
\log \frac{|x_{\check{S}}-x_{\check{S}'}|}{\min\{\delta_{\Omega_{\tilde{S}}}(x_{\check{S}}),\delta_{\Omega_{\tilde{S}}}(x_{\check{S}'})\}}+1
\\
\lec \log \frac{|x_{\check{S}}-x_{\check{S}'}|}{\min\{\ell(Q_{S}),\ell(Q_{S'})\}}+1
\\
 \lec \log \frac{|x_{\check{S}}-x_{\check{S}'}|}{\min\{\ell(Q),\ell(R)\}}+1
\end{multline*}

It is not hard to show that $|x_{\check{S}}-x_{Q}|\lec \ell(Q(\check{S}))\lec \ell(Q_{S})$ and similarly $|x_{\check{S}'}-x_{R}| \lec \ell(Q_{S'})$, and so
\[
|x_{\check{S}}-x_{\check{S}'}|
\lec |x_{Q}-x_{R}|+\ell(Q_{S})+\ell(Q_{S'})
\stackrel{\eqref{e:xqxrqsqs'claim}}{\lec}
\tau^{-2} |x_{Q}-x_{R}|.
\]
Thus, combining the estimate \eqref{e:the-ups} for the length of the chain from $x_{S}$ to $x_{\check{S}}$ and from $x_{S'}$ to $x_{\check{S}'}$ and the estimate for the length of the chain from $x_{\check{S}}$ to $x_{\check{S}'}$, we obtain a chain from $x_{S}$ to $x_{S'}$ of length at most 
\[
\log\frac{|x_{Q}-x_{R}|}{\min\{\ell(Q),\ell(R)\}}+1.
\]

Finally, we can connect $x_{Q}$ to $x_{S}$ and $x_{R}$ to $x_{S'}$ by chains of total length at most
\[
\lec \log \frac{\ell(Q(S))}{\ell(Q)}+\log \frac{\ell(Q(S')}{\ell(R)}+1
\lec \log \frac{\max\{\ell(Q(\check{S})),\ell(Q(\check{S'}))\}}{\min\{ \ell(Q),\ell(R)\}}+1
\]
but this is the term in \ref{e:the-ups} which we have already bounded.

\end{proof}

\begin{lemma}
\label{l:omegacorkscrews}
The domain $\Omega_{0}$ has the exterior corkscrew property. 
\end{lemma}

\def\bbeta{b\beta}
The proof is similar to \cite[Lemma 4.1]{AHMNT17}. We first recall some facts about UR sets. For a set $E$, $x\in E$, a hyperplane $P$, and $r>0$ define 
\[\bbeta_{E}(x,r,P)=r^{-1}\left(\sup_{y\in
E\cap B(x,r)}\dist(y,P)+ \sup_{y\in P\cap
B(x,r)}\dist(y,E)\right)\]
and then define
\[\bbeta_{E}(x,r)=\inf_{P}\bbeta_E(x,r,P)\]
where the infimum is over all $d$-dimensional hyperplanes
$P\subseteq\mathbb{R}^{d+1}$.

\begin{definition}\label{defi:BWGL}
We say that an $d$-dimensional ADR set $E$
satisfies the {\it bilateral weak geometric lemma} or {\it
BWGL} if, for each $\varepsilon>0$,
the set
\[
{B}_{\varepsilon}:=
\big\{(x,r):x\in E, r>0,\bbeta_E(x,r)\geq\varepsilon\big\}
\]
is a {\it Carleson set}, i.e., if we define
\[
\widehat{\sigma}(A)=\iint_{A}d\cH^d\frac{dt}{t},
\qquad A\subset E\times(0,\infty),
\]
then
\begin{equation}\label{e:carleson-sum}
\widehat{\sigma}
\left( B_{\varepsilon}\cap \big(B(x,r)\times (0,r)\big)\right)
\lec r^{d}
\end{equation} for all $x\in E$ and $0<r<\diam E$.
\end{definition}

\begin{theorem}[{\cite[Theorem 2.4, Part I]{DS2}}]\label{theor:UR-2}
A $d$-dimensional ADR set $E$ is uniformly rectifiable if and only if it satisfies the BWGL.
\end{theorem}

\begin{proof}[Proof of Lemma \ref{l:omegacorkscrews}]
Let $x\in \d\Omega_{0}$ and $r>0$. If $r\leq 2\delta_{\Omega}(x)$, then $x\in \d\lambda Q$ for some Whitney cube $Q\in \cW_{S}$ for some $S\in \tilde{\cF}$. Then it is not hard to see that, since $\Omega_{0}$ is a union of dilated Whitney cubes that we may find a large corkscrew point in $B(x,r)\backslash \Omega_{0}$. 

If $r> 2\delta_{\Omega}(x)$, then there is $y\in \d\Omega\cap B(x,r/2)$. Since $\d\Omega$ is uniformly rectifiable, by Theorem \ref{theor:UR-2}, \eqref{e:carleson-sum} holds, and so for any $\ve>0$ and $M>0$, we may find a ball $B(z,r')$ with $r/4>r'\gec_{\ve,M} r$ and $z\in B(y,r/4)\cap \d\Omega$ for which 
\begin{equation}
\label{e:beta}
 \bbeta_{\d\Omega} (z,Mr',P)<\ve.
\end{equation}
for some $d$-plane $P$.  By replacing $\ve$ with $2\ve$ if need be, we can assume $z\in P$. Let $x^{\pm}=z\pm \frac{r'}{2} e_{P}$, where $e_{P}$ is the unit normal vector to $P$.  

{\bf Claim:} At least one of $x^{+}$  and $x^{-}$ is in $\Omega_{0}^{c}$. Suppose instead that both were contained in $\Omega_{0}$. Then by uniformity, there is a $C$-cigar curve $\gamma\subseteq \Omega_{0}$ from $x^{+}$ to $x^{-}$. If $\zeta\in P\cap \gamma$, we must have $\zeta\in B(z,C|x^{+}-x^{-}|)\subseteq B(z,2Cr')$ and, since $\Omega_{0}\subseteq \Omega$, 
\begin{align*}
\dist(\zeta,\d\Omega)
& \geq \dist(\zeta,\d\Omega_{0})
\geq C^{-1}\min\{\ell(\zeta,x^{+}),\ell(\zeta,x^{-})\}\\
& \geq C^{-1} \min\{|\zeta-x^{+}|,|\zeta-x^{-}|\}
\geq \frac{r'}{2C}.
\end{align*}

If we pick $M>3C$ and $\ve<\frac{1}{4CM}$, then this is a contradiction since \eqref{e:beta} implies 
\[
\sup_{w\in P\cap
B(z,Mr')}\dist(w,\d\Omega)<M\ve r'<\frac{r'}{4C}.\] 
This proves the claim. 

Without loss of generality, we may assume $x^{+}\in \Omega_{0}^{c}$. Note that by \eqref{e:beta}, 
\[
\dist(x^{+},\d\Omega)\geq \dist(x^{+},P)-M\ve r'
\geq \frac{r'}{4}
\]
so $B(x^{+},r'/4)\subseteq  B(z,r')\backslash \d\Omega  \subseteq B(x,r)\backslash \d\Omega$. If $B(x^{+},r'/8)\cap \d\Omega_{0}=\emptyset$, then $B(x^+,r'/8)\subseteq B(x,r)  \backslash \Omega_{0}$
 is our desired exterior corkscrew ball. Otherwise, there is $z'\in B(x^{+},r'/8)\cap \d\Omega_{0}$, and by the above displayed inequality, $z'\not\in \d\Omega$, so $z'\in \d\lambda Q$ for some $Q\in \cW_{0}$. By case 1, we may find an exterior corkscrew point in $B(z',r'/8)\backslash \Omega_{0}$, and then use this as our exterior corkscrew ball for $B(x,r)$. 

\end{proof}

\begin{lemma}
The domain $\Omega_{0}$ has ADR boundary.
\label{l:omegadr}
\end{lemma}

\begin{proof}
Upper regularity follows from the Carleson packing condition and the fact that each $\Omega_{S}$ is ADR, and lower regularity follows from the interior and exterior corkscrew conditions. See for example \cite[Appendix A.3]{HM14}. 
\end{proof}

We have now established that $\Omega_{0}$ is uniform with exterior corkscrews and ADR boundary, that is, $\Omega_{0}$ is a CAD, and this finishes the proof of Lemma \ref{l:CAD}.

\section{The  proof of Theorem \ref{thmii}}
\label{s:proof-of-theorem}

We now prove Theorem \ref{thmii}. The implication (1) implies (2) is Theorem \ref{l:CAD}, and as mentioned before, (2) implies (1) is immediate, so (1) and (2) are equivalent, and (as mentioned in the introduction) (3) implies (1) follows by earlier work, so we will just show that (2) implies (3).  We will fist prove the $A_{\infty}$-property in cubes that are tops of stopping-time regions. Let $\delta>0$, $Q_{0}\in \{Q(S):S\in \cF\}$, $x\in \Omega\backslash MB_{Q_{0}}$ (where $M$ is as in Lemma \ref{l:mainhmlemma2}), and  $E\subseteq Q_{0}$ be so that
\[
\frac{\omega_{\Omega}^{x}(E)}{\omega_{\Omega}^{x}(Q_{0})}
< \ve
\]
where $\ve>0$ is a constant we will choose to be small later depending on $\delta$. By Lemma \ref{l:mainhmlemma2}, there is a corkscrew point $x_{0}$ for $B_{Q_{0}}$ so that 
\[
\omega_{\Omega}^{x_{0}}(E)\lec \ve. 
\]
Let $\Omega_{0}$ be from Lemma \ref{l:CAD} contain $x_{0}$ as a corkscrew point and such that
\[
\cH^{d}(Q_{0}\backslash \d\Omega_{0})<\frac{\delta}{2}\cH^{d}(Q_{0}).\]
Then by the maximum principle, 
\[
\omega_{\Omega_{0}}^{x_{0}}(E)
\leq  \omega_{\Omega}^{x_{0}}(E)\lec \ve.
\]
By the main result in \cite{DJ90}, $\omega_{\Omega_{0}}^{x_{0}}$ is an $A_{\infty}$-weight, and so for $\ve>0$ small enough, we can guarantee that
\[
\cH^{d}(\d\Omega_{0}\cap E)<\frac{\delta}{2}\cH^{d}(Q_{0})
\]
Thus, 
\[
\cH^{d}(E)<\delta \cH^{d}(Q_{0}).
\]

This proves $\omega_{\Omega}^{x}\in A_{\infty}(\cH^{d}|_{\d\Omega},Q_{0})$. \\

Now we prove the general result. Let $\alpha>0$, $B$ be an arbitrary ball centered on $\d\Omega$.  By rescaling, we can assume without loss of generality that $r_{B}=c_0$ where $c_0$ is as in Theorem \ref{t:Christ}, and then pick a nested sequence of maximal $\rho^k$-nets for $\d\Omega$ $X_k$ so that $x_{B}\in X_0$. In this way, there is a cube $Q_0\in \cD_{0}$ with center $\zeta_Q=x_{B}$, so $c_{0}B_{Q_{0}}=B$, and hence $B\cap \d\Omega\subseteq Q_{0}$. Let $x\in \Omega$ be so that $\dist(x,MB)\geq \alpha>0$. 

Let $\cF$ be stopping-time regions as in Lemma \ref{l:HMM}. Since the $\{Q(S):S\in \cF\}$ satisfy a Carleson packing condition, it is not hard to show that we many find finitely many $S_{1},...,S_{N}$ so that 
\begin{enumerate}
\item $Q(S_{i})\subseteq Q_{0}$,
\item $Q(S_{i})\cap Q(S_{j})=\emptyset$ when $i\neq j$,
\item $\ell(Q(S_{i}))< \rho  \ell(Q_{0})$,
\item $N\sim_{\delta'} 1$,
\item $\cH^{d}(Q_{0}\backslash \bigcup_{i=1}^{N}Q(S_{i}))<\delta'\cH^{d}(Q_{0})$
\end{enumerate}
where $\delta'$ is a small number we will fix shortly, and $\rho$ is chosen small enough, depending on $\alpha$ and $M$, so that $x\in \Omega\backslash MB_{Q(S_{i})}$ for all $S_{i}$.

Let $E\subseteq B\cap \d\Omega$ and suppose $\omega_{\Omega}^{x}(E)<\ve$. By the doubling property for $\omega_{\Omega}$,
\[
\omega_{\Omega}^{x}(E \cap Q(S_{i}))
\leq \omega_{\Omega}^{x}(E)
<\ve \omega_{\Omega}^{x}(B)
\lec_{N} \ve \omega_{\Omega}^{x}(Q(S_{i})).
\]
and since $\omega_{\Omega}^{x}\in A_{\infty}(Q(S_{i}))$ for each $i$,  and $\ve>0$ small enough (depending on $N$ and hence only on $\delta$), we then have 
\begin{align*}
\cH^{d}(E)
& \leq \cH^{d}(E\backslash \bigcup_{i=1}^{N} Q(S_{i}))
+\sum_{i=1}^{N} \cH^{d}(E\cap Q(S_{i}))\\
& \leq \delta' \cH^{d}(Q_{0})+ \delta' \sum_{i=1}^{N} \cH^{d}( Q(S_{i}))
<\delta' \cH^{d}(Q_{0})\\
& \lec \delta' \cH^{d}(B\cap \d\Omega )
\end{align*}
and so for $\delta'$ small enough, we have $\cH^{d}(E)<\delta  \cH^{d}(B\cap \d\Omega )$, as desired.

\appendix

\section{Hrycak's example}

Here we sketch the construction of Hrycak's example and how to use it to give a semi-uniform domain with UR boundary but without BPLS. We fix an integer $n$ and define a set $E\subseteq \bR^{2}$ using the method of Venician blinds. Let $E_{0}=I_{0}$ be the unit line segment. Now divide it into $n$ half-open sub-intervals and rotate each interval $\theta=\frac{2\pi}{n}$ radians counterclockwise around its left endpoint, call this new set $E_{1}$, see Figure \ref{f:hrycak}. Let the new intervals be called $I_{1},...,I_{n}$, ordered by where their left endpoint lies along the real line. 

\begin{figure}[!ht]
\includegraphics[width=340pt]{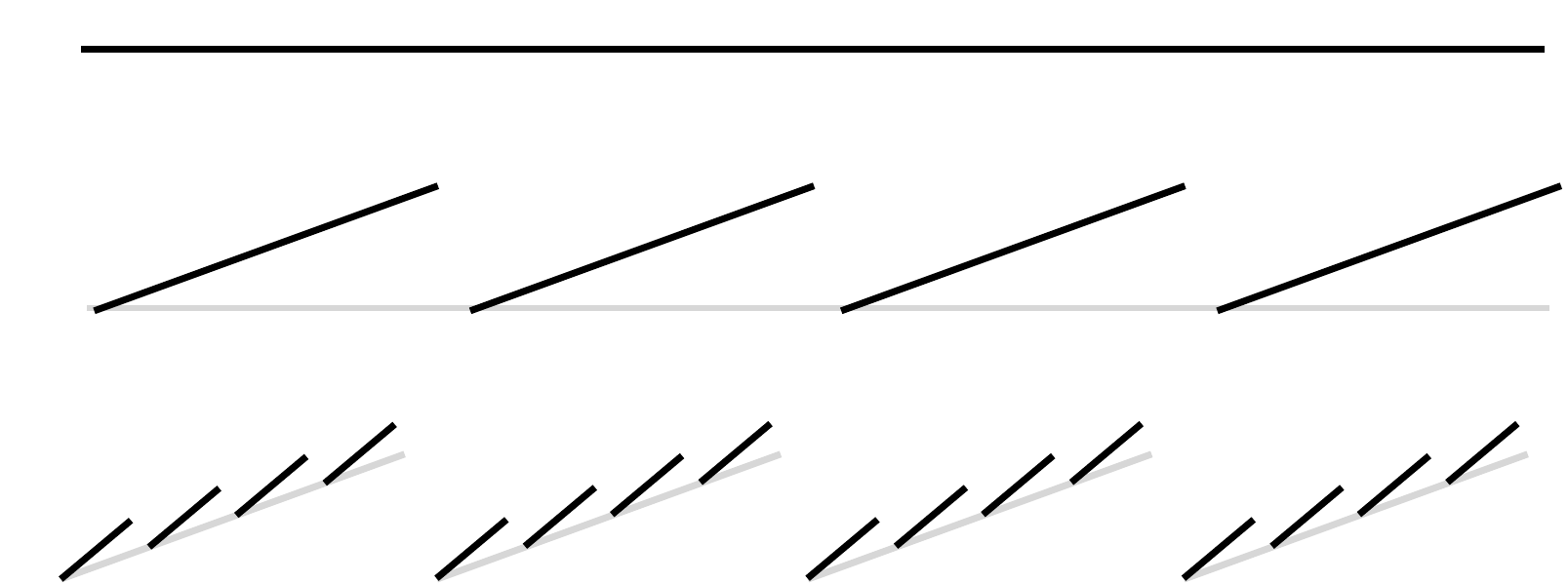}
\begin{picture}(0,0)(340,0)
\put(-10,125){$E_{0}$}
\put(-10,50){$E_{1}$}
\put(-10,0){$E_{2}$}
\end{picture}
\caption{Hrycak's example.}
\label{f:hrycak}
\end{figure}
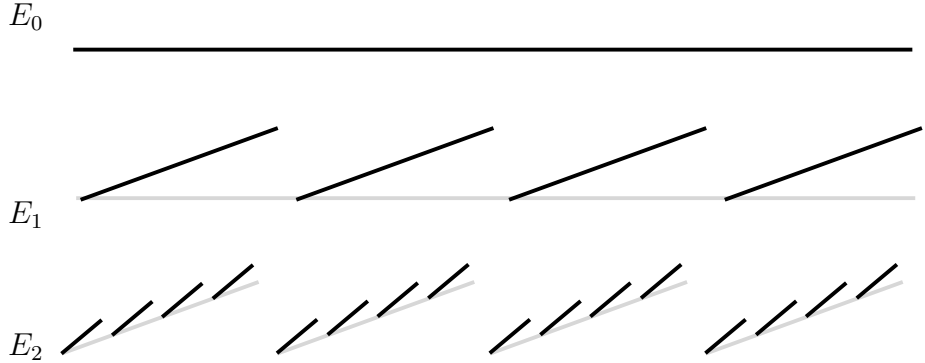

Repeat this on each new interval and so on for a total of $n$ times, defining sets $E_{2},...,E_{n}$ along the way, so that $E_{j}$ is a union of a set of $n^{j}$ half-open intervals $I_{\alpha}$ where $\alpha$ is a multi-index $\alpha=k_{1}....k_{j}$ and $k_{i}\in \{1,2,...,n\}$, and for $|\alpha|=j$ and each $i\in \{1,2...,n\}$, $I_{\alpha i}$ is the $ith$ subinterval of $I_{\alpha}$ rotated about its left endpoint. 

The resulting set $E=E_{n}$ is an ADR set whose orthogonal projections in the directions $0,\theta,2\theta,...,(n-1)\theta$ are at most $Cn^{-1}$ for some constant $C$. In particular, in order for a Lipschitz graph to intersect at least $c\%$ of it, say, the Lipschitz constant must be at least a constant times $n$. 

The remainder of the proof focuses on showing that $E^{c}$ is semi-uniform. We will do this by showing that $E$ bi-Lipschitz embeds into the real line. By a theorem of MacManus, we can extend the inverse of this map to a global bi-Lipschitz map of $\R^{2}$, and hence $E^{c}$ is the bi-Lipschitz image of a domain of the form $\R^{2}\backslash A$ where $A\subseteq \R$, and it is immediate that such domains are semi-uniform. Now we begin with the details.
%
%
%

We now construct a bi-Lipschitz embedding of $E$ into the real line. Let $K_{1},...,K_{n}$ be the intervals in the unit interval (ordered by their position on the real line) obtained by deleting $n-1$ equally spaced intervals of length $n^{-2}$. Then 
\[
|K_{j}|=(1-(n-1)\cdot \frac{1}{n^{2}})/n=: \frac{c_{n}}{n}.
\]
If $|\alpha|=j-1$ and we have defined $K_{\alpha}$, let $K_{\alpha 1},...,K_{\alpha n}$ be the intervals in $K_{\alpha}$ (ordered by their position on the real line) obtained by removing $n-1$ intervals of length $|K_{\alpha}|n^{-2}$. Note that since $ c_{n}^{n}\geq \frac{1}{2}$,
\[
|K_{\alpha i}|=|K_{\alpha}|\frac{c_{n}}{n}
=\frac{c_{n}^{j} }{n^{j}}\in [n^{-j}/2,n^{-j}]
\]
and the gaps between the $K_{\alpha i}$ are 
\[
|K_{\alpha}|n^{-2}  \in [n^{-j-2}/2,n^{-j-2}].
\]
It is not hard to show that, for $|\alpha|<n$, $\bigcup_{i=1}^{n} I_{\alpha i}$ bi-Lipschitz embeds into $I_{\alpha}$ by some $L$-bi-Lipschitz map $f_{\alpha}$ (with $L$ just some constant independent of $\alpha$) so that its image is $I_{\alpha}$ with $n-1$ equally spaced intervals of length $|I_{\alpha}|c_{n}n^{-1}$.

In particular, if $J_{\alpha 1},...,J_{\alpha n}$ are these intervals in order along $I_{\alpha}$, then $f_{\alpha}: I_{\alpha_{i}}\rightarrow J_{\alpha i}$ is affine and surjective, and
\[
\frac{|f_{\alpha}(I_{\alpha i})|}{|I_{\alpha i}|}
=c_{n}\]
We let $f_0$ be the map that takes $\bigcup_{i=1}^{n} I_{i}$ into the unit interval $I_{0}$. Hence, if $\alpha=\alpha_{1}\cdots \alpha_{k}$, define
\[
F_{\alpha}:=f_{0}\circ f_{\alpha_{1}}\circ f_{\alpha_{1}\alpha_{2}}\circ\cdots \circ f_{\alpha_{1}\alpha_{2}\cdots \alpha_{k-1}},\]
then $F_{\alpha}$ is an affine $c_{n}^{k}$-bi-Lipschitz map on $I_{\alpha}$, but since $c_{n}^{n}\sim 1$ for all $n$, we get that $F_{\alpha}$ is $C$-bi-Lipschitz on $I_{\alpha}$ for some universal constant $C$. 

 For $x\in I_{\alpha}\subseteq E$ with $|\alpha|=n$, we define $F(x)= F_{\alpha}(x)$. Moreover, we can do this in such a way that $F(I_{\alpha})=K_{\alpha}$ for all $|\alpha|=n$. 

We will now show this map is bi-Lipschitz. 

For $x\in I_{\beta}$ with $|\beta|=n$ and $\alpha=\beta_{1}\cdots \beta_{j}$ with $j<n$, define
\[
f^{\alpha}(x):=f_{\beta_{1}\cdots \beta_{j-1}} \circ\cdots \circ f_{\beta_{1}\cdots \beta_{n-1}} (x)\in I_{\alpha}\]
This defines a map  $f^{\alpha}:E\rightarrow  E_{j}$ so that $F_{\alpha}(f^{\alpha}(x))= F(x)$. 

If $x\in I_{\alpha}$ with $|\alpha|=n$, then if $\beta=\alpha_{1}...\alpha_{j}$
\begin{equation}
\label{e:lfax-x}
|f^{\beta}(x)-x|\lec \sum_{i=j+1}^{n} \frac{1}{n^{i+1}} \lec \frac{1}{n^{j+2}}.
\end{equation}
 
Now let $x\in I_{\alpha}$ and $y\in I_{\alpha'}$ with $|\alpha|=|\alpha'|=n$, we can assume $\alpha\neq \alpha'$. Let $\alpha_{0}$ be the largest common truncation of $\alpha$ and $\alpha'$, or equivalently, so that $K_{\alpha_{0}}$ is the smallest common ancestor to $K_{\alpha}$ and $K_{\alpha'}$, set $j=|\alpha_{0}|$ and $x_{0}=f^{\alpha_{0}}(x)$ and $y_{0}=f^{\alpha_{0}}(y)$. Observe that $F(x)=F_{\alpha_{0}}(x_{0}) \in K_{\beta}$ and $F(y)=F_{\alpha_{0}}(y_{0})\in K_{\beta'}$ where $K_\beta$ and $K_{\beta'}$ are the intervals created from $K_{\alpha_{0}}$ so that $\beta$ is a truncation of $\alpha$ and $\beta'$ is a truncation of $\alpha'$. Hence
\[
|x_{0}-y_{0}|\gec
|F_{\alpha_{0}}(x_{0})-F_{\alpha_{0}}(y_{0})|\geq \dist(K_{\beta},K_{\beta'})
\gec n^{-2}|K_{\alpha_{0}}|
\gec   n^{-j-2}.
\]

If $|x-y|\geq Mn^{-j-2}$, then
\[
||x-y| -|x_{0}-y_{0}||
\leq |x-x_{0}|+|y-y_{0}|
\stackrel{\eqref{e:lfax-x}}{\lec} n^{-j-2}
\lec M^{-1} |x-y|\]
and so for $M$ large enough, $|x-y|\sim |x_{0}-y_{0}|$. Hence, 
\[
|x-y|\sim |x_{0}-y_{0}|\sim |F_{\alpha_{0}}(x_{0})-F_{\alpha_{0}}(y_{0})| 
=|F(x)-F(y)|.\]

If $|x-y|< Mn^{-j-2}$, then, because $x\in I_{\alpha}$ and $y\in I_{\alpha'}$ and $|\alpha|=|\alpha'|\geq j+1$, for $n$ large enough,
\[
|x-y|\geq \dist(I_{\alpha},I_{\alpha'})
\gec n^{-j-2}\]
hence
\begin{align*}
|F(x)-F(y)|
& =|F_{\alpha_{0}}(x_{0})-F_{\alpha_{0}}(y_{0})|
\lec |x_{0}-y_{0}|
\leq |x_{0}-x|+|y_{0}-y|\\
& \lec n^{-j-2}\lec |x-y|
\end{align*}
and if $I_\beta$ and $I_{\beta'}$ are the intervals created from $I_{\alpha_{0}}$ so that $\beta$ is a truncation of $\alpha$ and $\beta'$ is a truncation of $\alpha'$, then 
\[
|x-y|
\leq Mn^{-j-2}
\lec M\dist(K_{\beta},K_{\beta'})
\leq M|F(x)-F(y)|.
\]

Thus, $F:E\rightarrow \bR$ is bi-Lipschitz with constant independent of $n$. In particular, $E$ is UR. By \cite[Theorem 1]{Mac95}, this bi-Lipschitz map can be extended to a bi-Lipschitz map $f$ of the plane onto itself. In particular, since $f(E)$ is a subset of the real line, hence if we define $\Omega=\bR^{2}\backslash E$, then $f^(\Omega)$ is a domain whose boundary is contained in the real line. It is thus easy to prove that $f(\Omega)$ is a semi-uniform domain, and hence so is $\Omega$ since $f$ is bi-Lipschitz (with a different semi-uniformity constant, but ultimately one that is independent of $n$). Thus, since we can pick $n$ as large as we wish, for any $L$ and $\ve>0$, we can now construct a semi-uniform domain $\Omega\subseteq \bR^{2}$ with ADR and UR boundary (both independent of $L$) so that the intersection of $\d\Omega$ with any Lipschitz graph of constant $L$ has measure at most $\ve$.

\frenchspacing

\bibliographystyle{alpha}

\newcommand{\etalchar}[1]{$^{#1}$}
\def\cprime{$'$}

\end{document}